\documentclass[12pt]{amsart}
\usepackage{amsthm,amssymb}
\usepackage{hyperref}
\usepackage{epsfig}
\usepackage{pstricks}
\usepackage{pst-plot}
\usepackage{tabmac}
\usepackage{pdflscape}

\usepackage[margin=1.1in]{geometry}

\newtheorem*{statement}{Theorem}
\newtheorem{theorem}{Theorem}[section]
\newtheorem{thm}[theorem]{Theorem}
\newtheorem{lemma}[theorem]{Lemma}
\newtheorem{lem}[theorem]{Lemma}
\newtheorem{proposition}[theorem]{Proposition}
\newtheorem{prop}[theorem]{Proposition}
\newtheorem{corollary}[theorem]{Corollary}

\newtheorem{conjecture}[theorem]{Conjecture}
\newtheorem{problem}[theorem]{Problem}

\theoremstyle{remark}
\newtheorem{example}{Example}
\newtheorem{remark}{Remark}

\def\a{{\mathbf{a}}}
\def\b{{\mathbf{b}}}
\def\i{{\mathbf{i}}}
\def\j{{\mathbf{j}}}

\def\h{\mathfrak h}

\def\LSym{\mathrm{LSym}}

\newcommand\shift[1]{\stackrel{\circlearrowleft}{#1}}

\def\sp{\mathrm{sp}}
\def\id{\mathrm{id}}
\def\rat{\mathrm{rat}}
\def\mult{\mathrm{mult}}
\def\wt{\mathrm{wt}}
\def\ep{\varepsilon}
\def\ph{\varphi}
\def\k{\kappa}
\def\sgn{\mathrm{sgn}}
\def\tS{{\tilde S}}
\def\xing{\mathrm{xing}}
\def\trop{\mathrm{trop}}

\def\XX{{\mathcal{X}}}
\def\uqsln{{U_q'({\mathfrak {\hat {sl_n}}})}}
\def\uqslm{{U_q'({\mathfrak {\hat {sl_m}}})}}
\def\fsln{{U_q({\mathfrak {{sl_n}}})}}
\def\fslm{{U_q({\mathfrak {{sl_m}}})}}
\def\sln{\mathfrak {\hat {sl_n}}}

\def\Z{{\mathbb Z}}
\def\C{{\mathbb C}}
\def\R{{\mathbb R}}
\def\N{{\mathcal {N}}}
\def\M{{\mathcal M}}

\newcommand\ip[1]{\langle #1\rangle}

\title{Crystals and total positivity on orientable surfaces}
 \author{Thomas Lam}\address
 {Department of Mathematics\\ University of Michigan\\ Ann Arbor\\ MI 48109 USA.}
 \date{\today}
 \email{tfylam@umich.edu}
 \urladdr{http://www.math.lsa.umich.edu/\~{ }tfylam}
 \thanks{T.L. was supported by NSF grant DMS-0652641 and DMS-0901111, and by a Sloan Fellowship.}
 \author{Pavlo Pylyavskyy}\address
{Department of Mathematics\\ University of Minnesota\\ Minneapolis\\ MN 55414 USA.}
 \email{pavlo@umich.edu}
 \urladdr{http://sites.google.com/site/pylyavskyy/}
 \thanks{P.P. was supported by NSF grant DMS-0757165.}
\begin{document}
\begin{abstract}
We develop a combinatorial model of networks on orientable surfaces, and study weight and homology generating functions of paths and cycles in these networks.  Network transformations preserving these generating functions are investigated. 

We describe in terms of our model the crystal structure and $R$-matrix of the affine geometric crystal of products of symmetric and dual symmetric powers of type $A$.  Local realizations of the $R$-matrix and crystal actions are used to construct a {\it double affine geometric crystal} on a torus, generalizing the commutation result of Kajiwara-Noumi-Yamada \cite{KNY} and an observation of Berenstein-Kazhdan \cite{BK3}.  

We show that our model on a cylinder gives a decomposition and parametrization of the totally nonnegative part of the rational unipotent loop group of $GL_n$. 
\end{abstract}
\maketitle
\tableofcontents
\section{Introduction}

\subsection{Networks on orientable surfaces}
The central objects of this paper are certain simple-crossing, vertex-weighted oriented networks $\N$ on an oriented surface $S$.  These networks may have sources and sinks on the boundary of $S$, but are conservative in the interior of $S$.  An example of a network we consider is in Figure \ref{fig:wire35}.
\begin{figure}[h!]
    \begin{center}
    \input{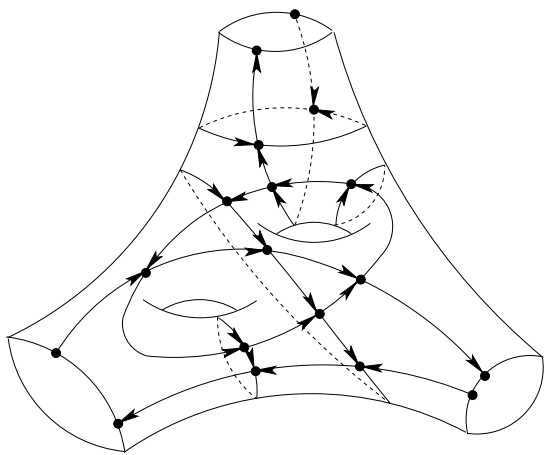}
    \end{center}
    \caption{A simple-crossing oriented network.}
    \label{fig:wire35}
\end{figure}

The first aspect of these networks we study are weight generating-functions of certain paths which we call {\it highway paths} in $\N$.  This definition relies crucially on the orientability of $S$.  We consider paths in $\N$ with endpoints on the boundary, giving us {\it boundary measurements}, imitating terminology of Postnikov \cite{Pos}, and also cycles in $\N$, giving us {\it cycle measurements}.  In both cases a crucial novelty of our work is to consider paths or cycles with specified {\it homology type}.

The second aspect that we study are {\it local transformations} or {\it local moves} of networks, that change a network $\N$ into another network $\N'$ on the same surface (see Section \ref{ssec:localmoves} for pictures of these).

These two aspects are motivated by the applications in the following table:

\begin{center}
\begin{tabular}{ |p{4cm}|p{5cm}|p{5cm}|}
\hline
&Measurements & Moves \\
\hline
Total positivity & Matrix entries, matrix coefficients & Changing factorizations or changing reduced words \\
\hline
Geometric crystals & $\ep$, $\ph$, weight function, energy function &  $R$-matrix, (crystal action) \\
\hline
Symmetric functions & elementary, homogeneous, power sum symmetric functions & $S_n$-action \\
\hline
Box-ball systems & integrals of motion & time evolution \\
\hline
\end{tabular}
\end{center}
Our general philosophy is that {\it measurements are preserved by moves}.
 
\subsection{Factorizations and parametrizations of totally positive matrices}

A matrix with real entries is {\it {totally nonnegative}} if all of
its minors are nonnegative.  There is a classical connection between totally nonnegative matrices and planar acyclic networks due to Linstr\"{o}m \cite{Li} and Brenti \cite{Br2}. Every totally nonnegative matrix $X$ can be represented  by a planar network so that the matrix entries are weight generating-functions of paths, and the minors of $X$ have an interpretation in terms of families of non-intersecting paths.  

Let $(U_n)_{\geq 0}$ denote the totally nonnegative part of the upper triangular unipotent subgroup $U_n$ of $GL_n(\mathbb R)$.  Lusztig \cite{Lus} established a decomposition $(U_n)_{\geq 0} = \sqcup_{w \in S_n} (U^w_n)_{\geq 0}$ of the totally nonnegative unipotent group into cells $(U^w_n)_{\geq 0}\simeq \R_{>0}^{\ell(w)}$.  The connection to networks, in the setting of this paper, is as follows: to each reduced word $\i$ one associates a network $N_\i$ such that boundary measurements give an element of $(U^w_n)_{\geq 0}$, and changing 
vertex weights give a parametrization of $(U^w_n)_{\geq 0}$.  Furthermore, if $\i$ and $\j$ are two reduced words for the the same $w \in S_n$, then $N_\i$ and $N_\j$ are related by local transformations.  From this point of view, our networks are closer to the wiring diagrams of \cite{BFZ, FZ} than the edge-weighted networks typically used in combinatorics.  Our most important local move is illustrated in the following example.

\begin{example} \label{ex:disk}
Figure \ref{fig:wire27} shows two networks in the disk which correspond to the two reduced words for the longest element of $S_3$.  
\begin{figure}[h!]
    \begin{center}
    \input{wire34.pstex_t} 
    \end{center} 
    \caption{}
    \label{fig:wire27}
\end{figure}
The corresponding element of the unipotent group is $$X = \left(
\begin{matrix}
1 & x+z & xy \\
0 & 1 & y \\
0 & 0 & 1 
\end{matrix}
\right) \in U_3$$
\end{example}

The Yang-Baxter move illustrated in Example \ref{ex:disk} is the network counterpart of the relation
\begin{equation} \label{eq:L}
 u_i(x) u_{i+1}(y) u_i(z) = u_{i+1}(yz/(x+z)) u_i(x+z) u_{i+1}(xy/(x+z))
\end{equation}
for the Chevalley generators of $U_n$.  This transformation is somewhat universal: Lusztig \cite{Lus} discovered that the tropicalization of this transformation, namely
\begin{equation}\label{eq:tropL}
(x,y,z) \mapsto (z+y - \min(x,z), \min(x,z), x+y-\min(x,z)),
\end{equation}
controlled the parametrizations of canonical bases of the quantum group of $U_n$.

A significant recent breakthrough was obtained by Postnikov \cite{Pos}, who studied (possibly non-acyclic) oriented networks in the disk, and developed the connections with total positivity of Grassmannians.  

In \cite{LP, LP3} we developed a theory of total positivity in the loop groups $GL_n(\R((t)))$.  Let $U_{\geq 0}$ denote the totally nonnegative part of the unipotent loop group.  The totally nonnegative part $(U_n)_{\geq 0}$ of the finite unipotent group is contained in $U_{\geq 0}$, so the following theorem generalizes the preceding discussion.

\begin{statement}[Theorems \ref{thm:whirlcurlBruhatdecomp}, \ref{thm:ratloop}, and \ref{thm:ntol}] \
\begin{enumerate}
\item[(a)]
Every $g(t) \in U_{\geq 0}$ which is rational can be represented by a network on the cylinder.
\item[(b)]
There is a decomposition 
$U^{\rat}_{\geq 0} = \sqcup_{a,b,w} U^{a,b,w}_{\geq 0}$ 
of the totally nonnegative part of the unipotent rational loop group into open-closed cells $U^{a,b,w}_{\geq 0}$ called \emph{whirl-curl-Bruhat cells}, where $a,b \in \Z_{\geq 0}$ and $w$ varies over the affine symmetric group $\tS_n$.
\item[(c)]
There is a network $N^{a,b,w}$ on the cylinder which parametrizes each $U^{a,b,w}_{\geq 0}$, up to the monodromy action of local transformations on vertex weights.
\end{enumerate}
\end{statement}

The monodromy action of (c) appears to be new in the study of networks and total positivity: this action is trivial for the reduced wiring diagrams of \cite{BFZ} or the reduced plabic networks of \cite{Pos}.  In the above theorem, the loop parameter $t$ of $GL_n(\R((t)))$ corresponds to a non-trivial homology basis element of the cylinder.  The operation of gluing two cylinders together along a boundary corresponds to the semigroup structure of $U^{\rat}_{\geq 0}$.

We speculate that there is a general notion of total positivity on any surface (see also Theorem \ref{thm:tpd}) such that finite networks with positive real weights give rise to exactly the rational totally positive points (see Section \ref{ssec:rat} and Conjecture \ref{conj:rat}).

\subsection{Crystals and networks}
{\it Crystal graphs} were invented by Kashiwara \cite{Kas} as combinatorial skeletons of representations of quantum groups.  Berenstein and Kazhdan \cite{BK,BK2} have developed a theory of {\it geometric crystals}, where the combinatorial structures of a crystal graph are replaced by rational functions and birational transformations of algebraic varieties.

In this paper we study products of the basic affine geometric crystal $X_M$ \cite{KNO} of $\uqsln$ (the geometric crystal corresponding to symmetric powers of the standard representation), and its dual $X_N$.

\begin{statement}[Theorem \ref{thm:whirlcurlcrystal} and Propositions \ref{P:epphrat} and \ref{P:energy}]
Let $X^\tau$ be a product of the affine geometric crystals $X_M$ and $X_N$.  There is a set $\XX^\tau$ of networks on the cylinder which can be identified with $X^\tau$, so that 
\begin{enumerate}
\item[(a)]
the functions $\ep_i$, $\ph_i$ together with the weight function and energy function are boundary measurements, or rational functions of the boundary measurements;
\item[(b)]
the $R$-matrix is computed by performing local transformations;
\item[(c)]
the crystal operator $e_i^c$ is computed by adding crossings and then performing local transformations.
\end{enumerate}
\end{statement}
The fact that the $R$-matrix commutes with the crystal structure is an example of our general philosophy that local transformations should preserve boundary measurements.  The above theorem also shows again the universality of \eqref{eq:L} which plays a crucial role in (b).  The tropicalization \eqref{eq:tropL} in turn plays a role in the combinatorial $R$-matrix of Kirillov-Reshetikhin crystals.

The topology underlying the networks plays a crucial role.  In our network realization of affine geometric crystals, the product of geometric crystals corresponds to gluing cylinders along their boundary.

Our investigations also lead to the definition of a {\it double affine geometric crystal} $\XX_{n,m}$, constructed from a network on a torus.  This geometric crystal has both a $\uqsln$-crystal structure and a $\uqslm$-crystal structure.  To a certain extent these crystal structures commute (Theorem \ref{thm:doubleaffine}); indeed the Weyl group action of the $\uqsln$-crystal  acts as the $R$-matrix for the $\uqslm$-crystal, and vice versa.  This generalizes Kajiwara, Noumi, and Yamada \cite{KNY}'s commuting birational actions of the affine symmetric group, Lascoux's double crystal \cite{Las}, and observations of Berenstein and Kazhdan \cite{BK2}.

\subsection{Measurements and moves}
The study of our oriented networks on surfaces is motivated by the connections to total positivity and geometric crystals, but we also develop the general theory of such networks.  Since the networks we study are not acyclic, the weight generating functions of paths and cycles we consider are a priori formal power series.  The first fundamental result we establish is

\begin{statement}[Theorem \ref{thm:fin}]
Boundary measurements are homogeneous polynomials in vertex weights.  Cycle measurements for a non-trivial homology class are homogeneous polynomials in vertex weights.
\end{statement}

This result relies on our notion of a ``highway path'' together with the homological\footnote{We remark that one could use homotopy instead of homology to build a somewhat different theory.} restrictions we impose.  As we have already alluded to, our measurements and transformations satisfy

\begin{statement}[Theorem \ref{thm:bm} and Proposition \ref{prop:torus}]
Boundary and cycle measurements are preserved by local transformations, and by a global action of a torus (Section \ref{ssec:torus}).
\end{statement}

The general philosophy motivating our combinatorial work on networks is that measurements and transformations should satisfy: 
\begin{enumerate}
\item measurements are invariant under the transformations;
\item two networks with the same measurements can be transformed into each other using the transformations (Problem \ref{pr:1});
\item measurements generate all the invariants of the monodromy and torus actions of the transformations on vertex weights (Problem \ref{pr:2}).
\end{enumerate}
In (3), the monodromy action refers to sequences of transformations of the networks which preserve the underlying graph but change the vertex weights (see for example Section \ref{sec:whurl}, and Corollary \ref{thm:SaSb}).  Two of our more refined goals are 
\begin{enumerate}
 \item describe explicitly the range of possible measurements, usually as vertex weights range over nonnegative real numbers;
 \item give explicit formulae for recovering the weights and/or the network from the measurements.
\end{enumerate}
For example, the totally nonnegative part of the unipotent rational loop group solves (1) for networks on the cylinder where sources and sinks are on different boundary components (see Theorem \ref{thm:ratloop}).  

We now illustrate all these features with the following example.

\subsection{Symmetric functions and loop symmetric functions}
Let us take an oriented network on a cylinder, consisting of one horizontal wire joining the two boundaries, and $n$ vertical cycles which loop around the cylinder.  The case $n=2$ is shown in Figure \ref{fig:wire31}. 

\begin{figure}[h!]
    \begin{center}
    \input{wire31.pstex_t}
    \end{center}
    \caption{}
    \label{fig:wire31}
\end{figure}

In this case, the boundary measurements are the complete homogeneous symmetric functions and the cycle measurements are the (rescaled) power sum symmetric functions of the vertex weights (see Examples \ref{ex:bm} and \ref{ex:cm}).  The monodromy action is an $S_n$-action on the $n$ vertical cycles.  This action generalizes to the case of many horizontal wires (see Section \ref{sec:whurl}).  We observe that
\begin{itemize}
\item the measurements, being symmetric functions of the variables, are preserved by the symmetric group action;
\item the measurements generate the ring of invariants by a theorem of Newton;
\item calculating the vertex weights from the measurements is equivalent to finding the roots of a polynomial equation, and generically any two solutions are connected by the action of the Galois group;
\item the variables are nonnegative real numbers if and only if the boundary measurements form a {\it {totally positive sequence}} -- this is part of the {\it {Edrei-Thoma theorem}}, see \cite{Br, Edr, Th};
\item describing the range of cycle measurements is equivalent to a version of the classical {\it {problem of moments}}, see \cite{ShT}.
\end{itemize}

As we shall see in Section \ref{sec:lsym}, the case of multiple horizontal wires on a cylinder leads to {\it loop symmetric functions}, introduced in \cite{LP}.

\subsection{Comparison of examples}
In the following table, the first column is the case of symmetric functions.  The second column is the generalization to multiple horizontal wires, discussed in Sections \ref{sec:whirlcurl} and \ref{sec:lsym}.  The third column concerns the case of wiring diagrams in a disk as in Example \ref{ex:disk}, with sources concentrated on one side of the disk and sinks on the other.
The fourth column concerns Postnikov's networks in a disk; the connection between Postnikov's networks and ours is discussed in Section \ref{ssec:disk}.  In the table, TNN is short for totally nonnegative.

\begin{landscape}
\newpage
\begin{tabular}{ |p{4cm} |p{4cm} |p{4cm} | p{4cm}| p{4cm}|}
\hline
 & cylinder with single horizontal wire, $n=2$ case in Figure \ref{fig:wire4} & cylinder with multiple horizontal wires in same direction & reduced wiring diagrams of permutations in a disk & Postnikov's reduced plabic graphs in a disk \cite{Pos}\\
\hline
  boundary and cycle measurements & complete homogeneous and power sum symmetric functions & complete homogeneous and power sum loop symmetric functions & matrix entries of the corresponding product of Chevalley generators, cf. Ex. \ref{ex:disk} & Pl\"ucker coordinates of a point on Grassmannian \\
\hline
local transformations & usual $S_n$-action on polynomial ring & $S_n$ action via whurl relations of Section \ref{sec:whurl} & Yang-Baxter moves \eqref{eq:L}& transformations of plabic networks \cite[Section 12]{Pos} \\
\hline
  measurements generate monodromy invariants? & yes (Newton) & yes (see Theorem \ref{thm:LP4} and \cite{LPLSym}) & monodromy is trivial & monodromy is trivial\\
\hline
  same measurements implies connected by transformations? & yes, Galois symmetry of roots of a polynomial & yes, \cite[Theorem 8.3]{LP} and Theorem \ref{thm:ntol} & yes, \cite{Lus, BFZ} & yes, \cite{Pos}\\
\hline
  range of measurements for nonnegative vertex weights & rational TP sequences (special case of Edrei-Thoma theorem \cite{Edr, Th, Br}) & TNN elements of rational unipotent loop group (Theorem \ref{thm:ratloop}) & TNN elements of the unipotent group & TNN elements of a Grassmannian\\
\hline
  solution for vertex weights from measurements & solving polynomial equation & ASW factorization \cite{LP} & chamber ansatz \cite{BFZ} & work of Talaska \cite{Tal2}\\
\hline
stratification & stratification by numbers and multiplicities of  roots& whirl-curl-Bruhat cells (Section \ref{sec:whirlcurl})& Lusztig's cell decomposition of $(U_n)_{\geq 0}$& positroid stratification of TNN Grassmannian \cite{Pos}\\
\hline
semigroup structure& gluing cylinders corresponds to multiplying polynomials&
gluing cylinders corresponds to product of loops into $GL_n$ & gluing networks corresponds to product of matrices & ?\\
\hline
\end{tabular}
\end{landscape}

\part{Boundary measurements on oriented surfaces}
\section{Networks and measurements} \label{sec:net}
\subsection{Oriented networks on surfaces}
Let $S$ be a compact connected oriented surface, possibly with boundary.  By an oriented network $N$ embedded in $S$, we mean a finite set of vertices in $S$, possibly lying on the boundary, and a finite set of directed edges (continuous paths) in $S$, such that the endpoints of each edge are vertices, and the interior of each edge does not intersect the boundary or any other edge.  We shall allow edges to be closed oriented loops, that is, to have no endpoints.  Thus each edge may have 0, 1, or 2 distinct endpoints.  Two networks $N$ and $N'$ are considered the same if they are combinatorially equivalent; that is, if $N$ and $N'$ can be obtained from each other by a continuous deformation which does not change the combinatorial structure.

We shall consider the class of oriented networks $N$ embedded in $S$, satisfying the following {\it simple-crossing} condition:
\begin{quote}
Each non-boundary vertex of $N$ has in-degree 2 and out-degree 2, and the incoming edges (resp. outgoing edges) are adjacent. 
\begin{figure}[h!]
    \begin{center}
    \input{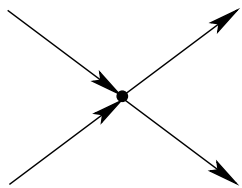}
    \end{center}
    \caption{A simple-crossing vertex.}
    \label{fig:wire1}
\end{figure}
\end{quote}
Our oriented networks will typically be vertex weighted: each interior vertex $v$ is given a  weight $x_v$. We shall often denote vertex weights just by $x,y,z, \ldots$ and write them next
to the vertices in figures. The vertex weights will in some cases take values in a specific field, such as the real numbers, and in other cases be considered indeterminates.  We shall also make the assumption that each boundary vertex has degree less than or equal to one.  Thus a boundary vertex is either a source, a sink, or an isolated vertex.

Vertex-weighted networks are usually denoted as $\N$; the underlying unweighted network of $\N$ is denoted $N(\N)$.  Many of our definitions depend only on the underlying unweighted network.  In these cases, we will say that $\N$ has a property if $N(\N)$ has that property.

\subsection{Polygon representation of oriented surfaces}\label{ssec:slice}
It is well known that oriented surfaces can be represented as polygons with certain edges identified.  Let $S$ be a compact oriented surface with $b$ boundary components and genus $g$.  Then $H_1(S,\Z) \simeq \Z^a$, where $a = 2g + \max(b-1,0)$ (\cite[Theorem 26C]{AS}).  Then $S$ can be cut into a polygon with edges identified using exactly $a$ slices $s_1,s_2,\ldots,s_a$ of $S$ (unless $S$ is a sphere).  We think of each slice $s_i$ as an oriented path on $S$, and assume that the slices intersect the edges of our network $\N$ transversally, and furthermore, do not intersect the vertices.

Let $p$ be a cycle in $\N$, or more generally a closed loop in $S$.  Suppose $p$ crosses the slice $s_i$ from right to left $r_i$ times and from left to right $l_i$ times.  Then choosing an appropriate basis for $H_1(S,\Z)  \simeq \Z^a$, we have $[p] = (r_1-l_1,r_2-l_2,\ldots,r_a-l_a) \in H_1(S,\Z)$.  In other words, the homology class of $p$ is completely determined by which slices it crosses and in which directions.

We shall always take homology with $\Z$-coefficients.

\subsection{Highway paths and cycles}
Suppose the edges incident to a simple-crossing vertex with weight $x$ are  $e_1,e_2,f_1,f_2$ in counterclockwise order, where $e_i$ are incoming and $f_i$ are outgoing.  Then we call
 the edges $e_1, f_1$ the {\it highway}, and $e_2,f_2$ the {\it underway}.  A {\it highway path} $p$ in $\N$ is a directed walk in $\N$, such that at each simple crossing, one must 
exit via the highway if one enters via the underway. In other words, in a highway path it is not allowed to exit through $f_2$ if one has entered through $e_2$, while the other three 
choices of the edges to enter and to exit are allowed. The {\it weight} $\wt(p)$ of a highway path $p$ is equal to the product of the vertex weights $x_v$ over all vertices $v$ (with 
multiplicities) traversed by $p$ such that $p$ entered and exited via the highway.
\begin{figure}[h!]
    \begin{center}
    \input{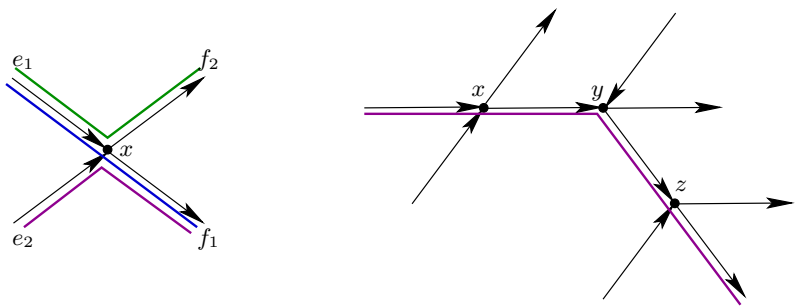}
    \end{center}
    \caption{The three ways a highway path can go through a vertex with weight $x$; a fragment of a highway path, contributing the factor $x z$ to the weight of the whole path.}
    \label{fig:wire2}
\end{figure}
Similarly, a {\it {highway cycle}} $p$ in $\N$ is a closed directed walk in $\N$, such that at each simple crossing, one must exit via the highway if one enters via the underway.
 The weight $\wt(p)$ of a highway cycle $p$ is equal to the product of the vertex weights over all vertices (with multiplicities) traversed by $p$ such that $p$ entered and exited via 
the highway. Note that in the case of a cycle to check the contribution of a specific vertex to the weight one must start walking around the cycle at another vertex.  Also note that we do not distinguish starting vertices of cycles.

\subsection{Boundary and cycle measurements} \label{sec:bac}
In the following, we assume that a network $\N$ on $S$, as in the previous subsections, has been fixed.

A cycle $p$ in the network $\N$ defines a homology class $[p] \in H_1(S):= H_1(S,\Z)$.  If $p$ and $q$ are two paths with the same starting and ending points, then we say that $p$ and 
$q$ are homologous if the cycle $p \cup q^{*}$ obtained from concatenating $p$ with the reverse of $q$ has trivial homology class.  In this case we write $p \sim q$.

Now suppose that we are given two boundary vertices $u$ and $v$, and a path $p$ from $u$ to $v$ to in $\N$.  Then we define the {\it boundary measurement} 
$$
M^{[p]} = \sum_{q: \; q \sim p} \wt(q)
$$
where the summation is over highway paths from $u$ to $v$ which are homologous to $p$.  
\begin{figure}[h!]
    \begin{center}
    \input{wire3.pstex_t}
    \end{center}
    \caption{}
    \label{fig:wire3}
\end{figure}

\begin{example}\label{ex:bm}
Figure \ref{fig:wire3} shows a network on a cylinder and a highway path $p$ between the boundary vertices. We have $\wt(p) = x$, and the whole boundary measurement is $M^{[p]} = x + y$.
 This is because there is exactly one other highway path of the same homology type, with weight $y$.
\end{example}

Suppose one has a cycle $p$ in $\N$.   The {\it multiplicity} $\mult(p)$ of $p$ is the maximum $k$ such that $p$ is obtained from repeating $k$ times some other (shorter) cycle $p'$.  We define the {\it cycle measurement} 
$$
M^{[p]} = \sum_{q: \; [q] = [p]} \frac{1}{\mult(q)}\wt(q)
$$
where the summation is over highway cycles with the same homology class $[q]$ as $[p]$.

\begin{example}\label{ex:cm}
 Figure \ref{fig:wire4} shows a network on a cylinder and a highway cycle $p$ that goes around the cylinder twice. We have $\wt(p) = \frac{1}{2} x^2$, since the multiplicity of this
 cycle is $\mult(p)=2$. The whole boundary measurement is $M^{[p]} = \frac{1}{2} x^2 + \frac{1}{2} y^2$, since there is exactly one other highway cycle of the same homology type, 
with weight $y^2$ and multiplicity $2$.
\end{example}

\begin{figure}[h!]
    \begin{center}
    \input{wire4.pstex_t}
    \end{center}
    \caption{}
    \label{fig:wire4}
\end{figure}

The summations in boundary and cycle measurements are a priori infinite sums.  We shall establish that they are in fact finite.

\subsection{Torus with one vertex}
To illustrate taking the measurements, let us consider the following simple network $\N$ on a torus, as shown in Figure \ref{fig:wire33}. 
\begin{figure}[h!]
    \begin{center}
    \input{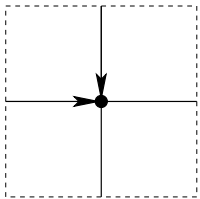}
    \end{center}
    \caption{A simple network on a torus.}
    \label{fig:wire33}
\end{figure}
It has one vertex and two closed cycle edges attached to it. Let $x$
be the weight of the vertex, and associate the horizontal and the vertical edges with the generators $(1,0)$ and $(0,1)$ (respectively) of the homology group $H_1(S,\Z) \simeq \mathbb Z^2$. Then any 
highway cycle in $\N$ consists of building blocks that have homology either $(1,1)$ or $(0,1)$. Therefore highway cycles of a fixed homology type are in bijection with {\it {necklaces}} with prescribed number of black and white beads,
corresponding to the $(1,1)$ and $(0,1)$ building blocks, respectively. Each necklace in addition has to be counted with coefficient equal to the inverse of the {\it {multiplicity}} of the necklace, 
where the latter is defined similarly to the multiplicity of a cycle. This gives us the following beginning of a table, where row and column tell us the homology type of the highway cycles we consider, and the 
entry tells us the corresponding cycle measurement.

\begin{center}
\begin{tabular}{c|c|c|c|c|c|c|c}
& $(-,0)$& $(-,1)$& $(-,2)$& $(-,3)$& $(-,4)$& $(-,5)$& $(-,6)$\\
\hline 
$(0,-)$& & $x$& $x^2/2$& $x^3/3$& $x^4/4$& $x^5/5$& $x^6/6$\\ 
\hline 
$(1,-)$& $0$& $1$& $x$& $x^2$& $x^3$& $x^4$& $x^5$\\ 
\hline 
$(2,-)$& $0$& $0$& $1/2$& $x$& $3x^2/2$& $2x^3$& $5x^4/2$\\ 
\hline 
$(3,-)$& $0$& $0$& $0$& $1/3$& $x$& $2x^2$& $10x^3/3$\\ 
\hline 
$(4,-)$& $0$& $0$& $0$& $0$& $1/4$& $x$& $5x^2/2$\\ 
\hline 
$(5,-)$& $0$& $0$& $0$& $0$& $0$& $1/5$& $x$\\ 
\hline 
$(6,-)$& $0$& $0$& $0$& $0$& $0$& $0$& $1/6$\\ 
\end{tabular}
\end{center}

For example, the entry $10x^3/3$ corresponds to four necklaces one can form from three black and three white beads. Three of those necklaces have multiplicity $1$ and contribute $x^3$ into the 
measurement, while one of them where black and white beads alternate has multiplicity $3$ and contributes $x^3/3$.  It is easy to see that the cycle measurement with homology type $(m,n)$ equals 
$\frac{1}{m} {{n-1} \choose {m-1}}$ if $n \geq m$, and $0$ otherwise.

\subsection{Flows and intersection products in homology}
The homology group $H_1(S)$ is a free abelian group which has an intersection pairing $\ip{.,.}$ defined as follows.  Given two cycles $p$ and $q$ in $S$ that intersect transversally, we define $\ip{p,q}$ to be \begin{align*}\ip{p,q}=&\#\{\text{intersections where $q$ crosses $p$ from the right}\}\\
&-\#\{\text{intersections where $q$ crosses $p$ from the left}\}
\end{align*}  We say that the {\it flow of $q$ through $p$} is equal to $\ip{p,q}$.  It is a basic fact \cite[Section 31]{AS} in the topology of surfaces that $\ip{p,q}$ depends only on the homology classes of $p$ and $q$.    More generally, one can define the flow of $q$ through $p$, where $q$ is a $1$-chain with boundary lying in the boundary $\partial S$, and $p$ is a cycle in $S$.  Again this flow does not change if $p$ (resp. $q$) are replaced by homologous cycles (resp. chains).  

Let us now describe these flows combinatorially for subnetworks of $N$.  Let $p$ be a cycle in $N$, and let $q$ be a subnetwork of $N$ (possibly with multiplicity) which is conservative in the interior of $S$.  In other words, for every vertex $v \in q$ in the interior of $S$, the number of incoming edges is equal to the number of outgoing edges.  Thus $q$ could be a cycle, or a path starting and ending on the boundary, or $N$ itself.  Then the {\it flow of $q$ through $p$} is defined as follows.  Let $p = (e_1,v_1,e_2,v_2,\ldots, e_r,v_r)$ be the sequence of directed edges and vertices in $p$.  For each triple $(e_i,v_i, e_{i+1})$ of adjacent edges, define 
\begin{align*}
a_i& =  \#\{\text{edges of $q$ entering $v_i$ to the right of $e_i$ and $e_{i+1}$}\}\\
&-\#\{\text{edges of $q$ exiting $v_i$ to the right of $e_i$ and $e_{i+1}$}\}\\
&-\#\{\text{edges of $q$ entering $v_i$ to the left of $e_i$ and $e_{i+1}$}\}\\
&+\#\{\text{edges of $q$ exiting $v_i$ to the left of $e_i$ and $e_{i+1}$}\}.
\end{align*}
Note that edges of $q$ which are equal to $e_i$ or $e_{i+1}$ are considered neither to the left, nor to the right, and thus not counted.  Then the flow $\ip{p,q}$ of $q$ through $p$ is equal to $\sum_i a_i/2$.  
\begin{figure}[h!]
    \begin{center}
    \input{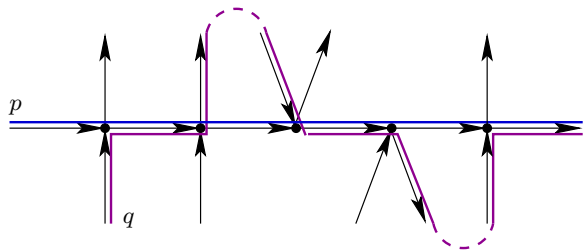}
    \end{center}
    \caption{The contribution of the shown fragments of $q$ and $p$ to $\ip{p,q}$ is $(1+1-1-1+1)/2 = 1/2$.}
    \label{fig:wire5}
\end{figure}

The main fact connecting homology and flows we will need is the following Lemma, which follows easily from the definitions.
\begin{lemma}\label{L:cons}
Suppose that $p$ is a cycle in $N$ and $q$ is a subnetwork of $N$ conservative in the interior of $S$.  Then $\ip{p,q} = 0$ if $[p]=0 \in H_1(S)$.
\end{lemma}

We say that a highway path or cycle $q$ {\it strongly contains} a highway cycle $$p = (e_1,v_1,e_2,v_2,\ldots, e_r,v_r),$$  if $q$ contains a subwalk of the form $v_i,e_{i+1},v_{i+1},\ldots,e_r,v_r,e_1,\ldots,v_{i-1},e_i,v_i,e_{i+1}$.  We say that $q$ strongly contains a collection of highway cycles if it strongly contains each of the cycles.

\begin{lemma}\label{lem:strongcontain}
Suppose $q$ is either a highway path in $N$ starting and ending on the boundary, or a highway cycle with non-trivial homology class.
Then $q$ does not strongly contain a collection of highway cycles $p_1,p_2,\ldots,p_r$ in $N$ such that $[p_1]+[p_2]+\cdots+[p_r] = 0 \in H_1(S)$.  
\end{lemma}
\begin{proof}
Suppose such a collection $p_1,\ldots,p_r$ exists.  Let us consider the flow of $N$ through $p_i$.  Since $p_i$ is a highway cycle, at each vertex the contribution to the flow is either $0$ or $1$.  Thus the flow $\ip{p_i,N}$ is nonnegative.  But $\sum_i \ip{p_i,N} = 0$ by Lemma \ref{L:cons}, since by assumption $N$ is conservative in the interior of $S$. It follows that at each vertex $v$ on each $p_i$ one has the following local situation: the edges belonging to $p_i$ are adjacent when viewed in the cyclic order around $v$.  Thus each vertex $v$ on $p_i$ either turns left or turns right.  

Now let us consider the flow $\ip{p_i,q}$ of $q$ through $p_i$.  Let $(e,v,e')$ be two consecutive edges of $p_i$ where $p_i$ turns left at $v$.  Each time $q$ arrives at $v$, one of the following happens:  (1) $q$ bounces off $p_i$ (that is, the four edges used are all distinct, and locally the two paths are tangent), or (2) $q$ follows $p_i$ through $v$ using both $e$ and $e'$, or (3) $q$ enters $v$ through $e$, and exits through another edge.  In case (3), one has a flow of $-1/2$.  If $p_i$ turns right at $v$, we have the cases: (1) $q$ bounces off $p_i$, or (2) $q$ follows $p_i$ through $v$ using both $e$ and $e'$, or (3) $q$ exits $v$ through $e'$, but enters through an edge not equal to $e$.  In case (3), one has a flow of $-1/2$.  Thus $\ip{p_i,q}$ is nonpositive.  But $\sum_i  \ip{p_i,q} = 0$, since $q$ is conservative.  It follows that case (3) can never occur.

Since it is assumed that $q$ strongly contains each of the $p_i$, it follows that $q$ must be a highway cycle which consists of traversing $p_i$ multiple times.  But then the $p_i$ must coincide as well up to multiplicity.  Since $[p_1]+[p_2]+\cdots+[p_r] = 0 \in H_1(S)$ it follows that $[q] = 0$, contradicting the assumption that $q$ has non-trivial homology.
\end{proof}

\begin{lemma}\label{lem:cone}
Let $N$ be a finite oriented simple-crossing network in $S$.  For each highway path $q$ in $N$, either starting and ending on the boundary, or a highway cycle with non-trivial homology class, we define a (possibly empty) semigroup $\Gamma_q \subset H_1(S) \simeq \Z^a$ by taking as generators the classes $[p_1],[p_2],\ldots,[p_r]$ of highway cycles strongly contained in $q$.  Then
\begin{enumerate}
\item
The $0$-vector does not lie in $\Gamma_q$.  Thus there exists $w_q \in \Z^a$ so that $(w_q,p_i) > 0$ for each $i$.  Here $(.,.)$ denotes the usual pairing on $\Z^a$.
\item
Only finitely many semigroups $\Gamma_q$ can be obtained in this way, as we let $q$ vary.
\end{enumerate}
\end{lemma}
\begin{proof}
(1) follows immediately from Lemma \ref{lem:strongcontain} (note that the $p_i$ in Lemma \ref{lem:strongcontain} can repeat).  To establish (2), let us say that a highway cycle $p$ is {\it irreducible} if it does not strongly contain a shorter highway cycle.  If a highway cycle $p$ strongly contains another highway cycle $p'$ then $p \setminus p'$ is still a highway cycle, and we have $[p] = [p']+[p \setminus p']$.  It follows that $\Gamma_q$ is generated by the classes of the irreducible highway cycles strongly contained in $q$.  But a highway cycle $p$ is irreducible if and only if no edge is repeated.  It follows that there are only finitely many irreducible highway cycles in $N$, and thus only finitely many possible $\Gamma_q$.
\end{proof}

\subsection{Polynomiality}
Let $\N$ be a vertex-weighted simple-crossing oriented network in $S$.
\begin{theorem}\label{thm:fin}
The boundary measurements $M^{[p]}$ are always homogeneous polynomials in the vertex weights.  The cycle measurements $M^{[p]}$ are homogeneous polynomials in the vertex weights except possibly when $[p] = 0 \in H_1(S)$.
\end{theorem}

\begin{proof}
Let $N = N(\N)$.  We must establish that there are finitely many highway paths (resp. cycles) $q$ homologous to $p$.  Suppose otherwise.  Since the network $N$ is finite, this implies that 
there are arbitrarily long $q$ homologous to $p$.  

First we claim that if $N'$ is a finite simple-crossing oriented network in a disk, then any sufficiently long highway path $q'$ in $N'$ strongly contains a highway cycle.  If $q'$ is very long then we can find an edge $e \in N'$ which occurs at least twice in $q'$.  The part of $q'$ between these two occurrences is the desired highway cycle.

Let us cut $S$ with a number of slices as in Section \ref{ssec:slice}, so that after slicing one obtains a disk from $S$.  Let the sliced edges in $N$ be $e_1,e_2,\ldots,e_k$, and set $v_i \in H_1(S)$ to be the homology class of a cycle which crosses the slices in the same way as $e_i$.  If $r$ is a path in $N$, we let the 
word of $r$ be $u(r) = e_{i_1} e_{i_2} \cdots e_{i_\ell}$ if $r$ encounters the sliced edges $e_{i_1}, e_{i_2},\ldots$ in that order.  If $r$ is a cycle then
 $[r] = v_{i_1} + \cdots + v_{i_\ell} \in H_1(S)$.  We shall also define $[r]$ by the same formula when $r$ is a walk.

By the earlier claim, if $q$ is a long highway path, then the word $u(q)$ of $q$ must also be arbitrarily long (otherwise it would contain a contractible highway cycle, contradicting
 Lemma \ref{lem:strongcontain}).  We show that this is impossible. Let $w = w_q$ be the vector of Lemma \ref{lem:cone}(1).  We now apply Lemma \ref{lem:extremal} below. 
The constant $C$ in Lemma \ref{lem:extremal} can be chosen to not depend on $q$, since by Lemma \ref{lem:cone}(2) there are only finitely many choices for the semigroups $\Gamma_q$, 
and thus finitely many choices for $w_q$.  Lemma \ref{lem:extremal} shows that if the length of $u(q)$ is longer than $b f(k,C)$ for some integer $b \geq 1$, then $([q],w) \geq b$. 
 But by Lemma \ref{lem:cone}(2) the maximum $\max_q([p],w_q)$ is finite (as one can pick finitely many distinct $w_q$), so the length of $u(q)$ is bounded which in turn implies the 
length of $q$ is bounded.

It follows that the claimed measurements are polynomials.  Next we observe that the
number of highway crossings in a highway path (or cycle) $p$ is equal to $\ip{p,N}$.  Since the latter
depends only on the homology class of $p$, we obtain that $M^{[p]}$ is homogeneous of degree $\ip{p,N}$.
\end{proof}

Recall that $(.,.)$ denotes the usual pairing $\Z^a \times \Z^a \to \Z$.

\begin{lemma}\label{lem:extremal}
For each integer $k \geq 1$ and $C \in \R_{>0}$ there exists a constant $f(k,C)$ such that the following holds.
Suppose $u = e_{i_1} e_{i_2} \cdots e_{i_\ell}$ is a word in the alphabet $e_{1},e_{2},\ldots,e_{k}$ and $v_1,v_2,\ldots,v_k \in \Z^a$ are any non-zero vectors, 
and $w \in \Z^a$ satisfies the conditions
\begin{enumerate}
\item
$(v_{i_s} + v_{i_{s+1}}+  \cdots + v_{i_{t - 1}},w) >0$ whenever $i_s = i_t$;  
\item
$|(v_{i},w)| < C$ for each $i \in \{1,2,\ldots,k\}$.
\end{enumerate}
Then either $\ell < f(k,C)$ or $(w,v_{i_1} + \cdots +v_{i_\ell}) > 0$.
\end{lemma}
\begin{proof}
If $x = e_{j_1}\cdots e_{j_r}$ is a word, we set $(x,w) = (v_{j_1}+\cdots + v_{j_r},w)$.

The claim is clearly true for $k = 1$.  We proceed by induction, assuming that $f(k-1,C)$ is known.  Suppose $\ell > k(2f(k-1,C)C+C+1)$.  Let $u$ be a word of length $\ell$ satisfying the conditions of the lemma.  By relabeling we may assume that $e_{k}$ occurs at least $(2f(k-1,C)C+C+1)$ times in $u$.  Suppose $u = x e_k x' e_k x''$ where $x$ and $x''$ are words not containing $e_k$.  Then by the inductive hypothesis, $(x,w) > -f(k-1,C)C$ and $(x'',w) > -f(k-1,C)C$.  But then $(u,w) > -2f(k-1,C)C + (2f(k-1,C)C+C) -C = 0$, using condition (1) applied to pairs of $e_k$ with no occurrences of $e_k$ in between.  Thus $f(k,C) = k(2f(k-1,C)C+C+1)+1$ works.
\end{proof}

\subsection{Rationality} \label{ssec:rat}
Let $\N$ be a vertex-weighted simple-crossing oriented network in $S$.  Suppose a basis for $H_1(S,\Z) \simeq \Z^a$ has been chosen.  If $$r = (r_1,r_2, \ldots, r_a) \in H_1(S,\Z),$$ we write 
$t^r$ for the monomial $t_1^{r_1}t_2^{r_2}\cdots t_a^{r_a}$.  In the sequel, we shall consider generating functions in the $t^r$ which will be compatible with change of basis of 
$H_1(S,\Z)$.

Let $p$ be a path between boundary vertices $u$ to $v$.  We consider the following generating function
$$
M_{u,v}(t) = \sum_q \wt(q) t^{[q \cup p^*]} = \sum M^{[q]} t^{[q \cup p^*]}
$$
where the first sum is over all highway paths $q$ from $u$ to $v$, and the second sum is over all homology classes of paths from $u$ to $v$.  The generating function depends on the
 ``basepath'' $p$; but changing $p$ to $p'$ will only change $M_{u,v}(t)$ to $t^{[p \cup (p')^*]}M_{u,v}(t)$. 

\begin{prop}\label{prop:rational}
$M_{u,v}(t)$ is a rational function in the vertex weights and in the homology variables $t_1^{\pm1},\ldots,t_a^{\pm1}$.
\end{prop}

\begin{proof}
There exists an edge-weighted network $G(\N)$ on $S$, with the same boundary vertices as $\N$, such that the measurements $M_{u,v}$ count arbitrary paths from $u$ to $v$ (with no 
highway condition) in $G(\N)$, where paths are now weighted by total edge weight.  The network $G(\N)$ is obtained by changing the network locally at every interior vertex as 
shown in Figure \ref{fig:wire7}.  If $z$ is the weight assigned to the original vertex, the edges of $G(\N)$ are assigned weights $z$ and $1$ as shown.  
\begin{figure}[h!]
    \begin{center}
    \input{wire7.pstex_t}
    \end{center}
    \caption{}
    \label{fig:wire7}
\end{figure}

To keep track of the homology classes of these paths, one slices $S$ into a disk as in Section \ref{ssec:slice}, and assigns a homology weight to each sliced edge according to which 
cuts the edge crosses, and in which direction.    Then up to basepath normalization, the homology class of a path $p$ in $G(\N)$ is equal to the product of homology weights of the 
slices it traverses.

Thus the generating function $M_{u,v}(t)$ is equal to a matrix entry in $(I-A)^{-1} = I + A + A^2 +\cdots$ where $A$ is the homology enriched adjacency matrix of $G$: edge weights 
have been multiplied with the corresponding homology weights.  Since all the entries of $A$ are polynomials in the vertex weights of $\N$, and the homology variables 
$t_1^{\pm1},\ldots,t_a^{\pm1}$, the result follows.
\end{proof}

\begin{example}
Consider the cylindric network in Figure \ref{fig:wire4}. Make a cut along the top of the cylinder and let the homology variable $t$ record each crossing of this cut.  
The edge-weighted network $G(\N)$ is in Figure \ref{fig:wire9}. 
\begin{figure}[h!]
    \begin{center}
    \input{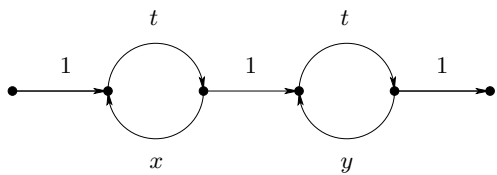}
    \end{center}
    \caption{The graph $G(\N)$ obtained from the network $\N$ in Figure \ref{fig:wire4}.}
    \label{fig:wire9}
\end{figure}

Then one computes the adjacency matrix and the generating function
$$A = \left(
\begin{matrix}
0 & 1 & 0 & 0 & 0 & 0\\
0 & 0 & t & 0 & 0 & 0\\
0 & x & 0 & 1 & 0 & 0\\
0 & 0 & 0 & 0 & t & 0\\
0 & 0 & 0 & y & 0 & 1\\
0 & 0 & 0 & 0 & 0 & 0
\end{matrix}
\right); \;\;\;\;\; M_{u,v}(t) = \frac{t^2}{(1-tx)(1-ty)}.
$$
\end{example}

Let us also define the generating function for cycle measurements
$$
M(t) = \sum_q \frac{1}{\mult(q)}\wt(q) t^{[q]}
$$
where the sum is over all highway cycles not homologous to 0.

\begin{prop}
There exists a rational function $f$ in the vertex weights, homology variables $t_1^{\pm1},\ldots,t_a^{\pm1}$, and an additional variable $\zeta$, such that $M(t)$ is obtained from the specialization
 at $\zeta=1$ of the integral of $f$ with respect to $\zeta$ by removing the constant term. 
\end{prop}
\begin{proof}
The removal of the constant term is simply to match with our convention to omit highway cycles which are homologically trivial.

Let $G(\N)$ be the graph in the proof of Proposition \ref{prop:rational}, and $A$ the homology enriched adjacency matrix.  If we just took ${\rm trace}((I-A)^{-1})$ we would enumerate 
cycles with marked starting vertices in $G(\N)$.  Our cycle measurements count cycles where no starting vertex is specified.  

Consider the expression $f(\zeta) = {\rm trace}((I-\zeta A)^{-1})$ which enumerates cycles in $G(\N)$ with marked starting vertices, and such that the power of $\zeta$ keeps track of 
the length of a cycle.  If $p$ is a cycle of multiplicity $k$, and which is a $k$-th power of a cycle $p'$ which is of length $r$, then $p$ will be counted $r$ times in $f(\zeta)$.  
Thus the cycle $p$ will have contribution $1/k$ in $\int f(\zeta) d\zeta$.  So $\int f(\zeta) d\zeta|_{\zeta=1}$ is equal to $M(t)$.
\end{proof}

\begin{example}
 In the previous example, we have $${\rm trace}((I-\zeta A)^{-1}) = 6 + 2tx\zeta^2+ 2ty\zeta^2+2t^2x^2\zeta^4+2t^2y^2\zeta^4+2t^3x^3\zeta^6+2t^3y^3\zeta^6+\ldots.$$ After integrating and 
setting $\zeta=1$, and removing the constant term, we get $$M(t) = tx+ ty+\frac{1}{2}t^2x^2+\frac{1}{2}t^2y^2+\frac{1}{3}t^3x^3+\frac{1}{3}t^3y^3+\ldots.$$
\end{example}

\subsection{Snake paths} \label{ssec:snake}
A {\it snake path} in $N$ is a path which starts and ends on the boundary, and turns at every interior vertex.  A {\it {snake cycle}} in $N$ is a cycle which turns at every interior vertex.  Every edge of $N$ belongs to a unique snake
path or snake cycle.  Note that snake paths or cycles can revisit vertices, and so may contain all four edges incident to some vertex.
\begin{figure}[h!]
    \begin{center}
    \input{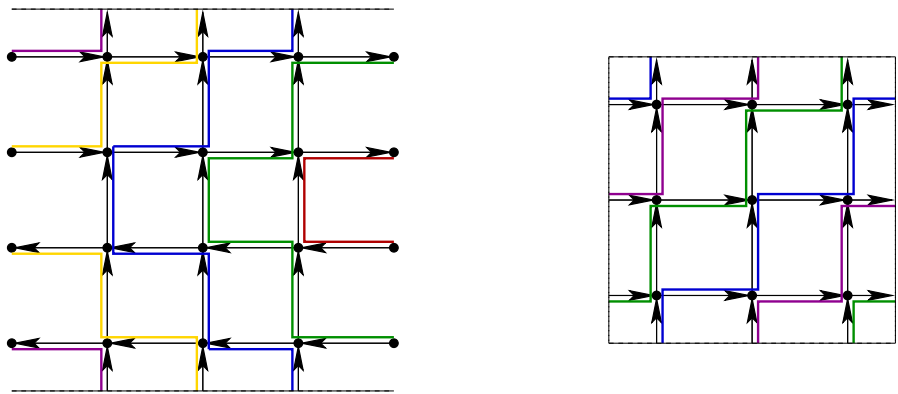}
    \end{center}
    \caption{Two networks, one on a cylinder and one on a torus, with snake paths/cycles shown.}
    \label{fig:wire6}
\end{figure}

\begin{lemma}\label{lem:snake}
Let $p$ be either
a highway path in $N$, starting and ending on the boundary, or a highway cycle in $N$.  Then
$\ip{p,q} \geq 0$ for every snake path or cycle $q$ in $N$. Furthermore, if $\ip{p,q} = 0$ for 
every snake path or cycle $q$, then $p$ is itself a snake path or cycle.
\end{lemma}
\begin{proof}
At each vertex $v$ of $p$, the path either follows the snake path or cycle $q$ it is currently on, or leaves $q$ to go on a different snake path $q'$.  In the latter case, one obtains a contribution of $1/2$ to both $\ip{p,q}$ and $\ip{p,q'}$.  If $\ip{p,q} =0$ for every snake path or cycle $q$, then $p$ can never leave a snake path, and thus must be equal to a snake path or cycle.
\end{proof}

\section{Local transformations of topological networks}\label{sec:local}
\subsection{Local moves}\label{ssec:localmoves}
We consider the following local transformations on simple-crossing, vertex-weighted oriented networks $\N$:
\begin{enumerate}
 \item (YB) Yang-Baxter move (Figure \ref{fig:wire8}): this is the usual ``braid move'' for wiring diagrams.  The rational transformations of vertex weights occur for example in the theory of total positivity \cite{Lus}.
\begin{figure}[h!]
    \begin{center}
    \input{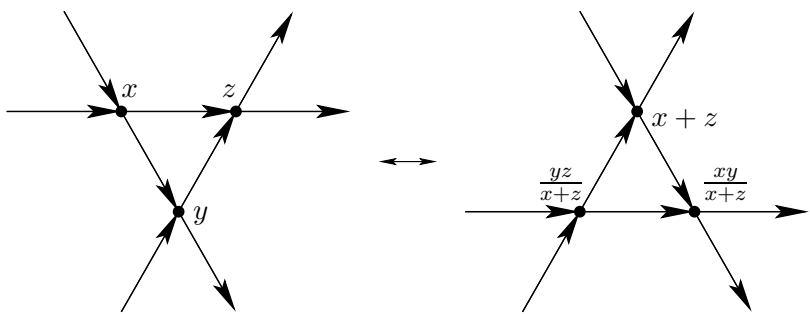}
    \end{center}
    \caption{Yang-Baxter move with transformation of vertex weights shown.}
    \label{fig:wire8}
\end{figure}

\item (CR) Cycle removal (Figure \ref{fig:wire10}): an oriented cycle bounding a disk, and not containing any part of $\N$ in its interior, can be removed.  Note that the vertices of the cycle are removed, and some edges of the rest of the network are glued together.  It is thus possible to create edges with no endpoints in this way.
\begin{figure}[h!]
    \begin{center}
    \input{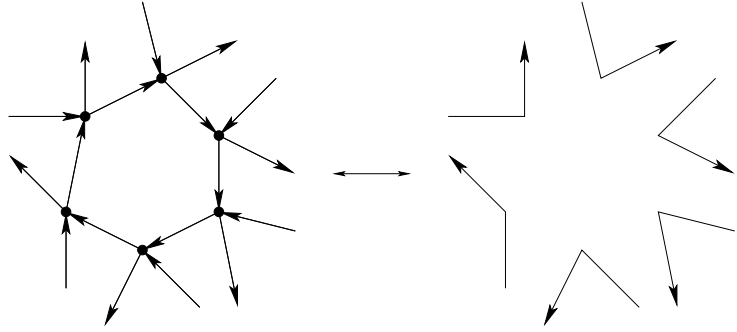}
    \end{center}
    \caption{Cycle removal move.}
    \label{fig:wire10}
\end{figure}

\item (XM) Crossing merging (Figure \ref{fig:wire11}): in a local configuration which looks like two wires (oriented in the same direction) crossing twice, the crossings can be merged to form a single crossing.

 \item (XR) Crossing removal (Figure \ref{fig:wire11}): a vertex with a 0 weight may be removed.
\begin{figure}[h!]
    \begin{center}
    \input{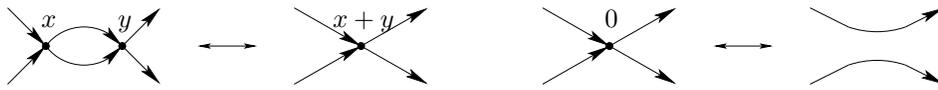}
    \end{center}
    \caption{Crossing merging and crossing removal moves with vertex weights shown.}
    \label{fig:wire11}
\end{figure}

\item (WC) Whirl-curl relation (Figure \ref{fig:wire12}): this move can be performed on a local part  of the surface that looks like a cylinder with two vertical wires of opposite orientation wrapping around the cylinder and $n$ horizontal wires crossing the vertical wires.  The move swaps the two vertical wires.  The associated rational transformation of the vertex weights is given by:  
\begin{equation}\label{eq:whirlcurl}
x'^{(i)} = \frac{y^{(i)}(x^{(i-1)}+y^{(i-1)})}{x^{(i)}+y^{(i)}}, \;\;\;\; y'^{(i)} = \frac{x^{(i)}(x^{(i-1)}+y^{(i-1)})}{x^{(i)}+y^{(i)}}
\end{equation}
if in the initial ($x,y$) configuration the horizontal wire first encounters a vertical wire crossing from right to left (and then encounters one crossing from left to right).  In the opposite direction, the rule is 
\begin{equation}\label{eq:whirlcurl2}
x^{(i)} = \frac{y'^{(i)}(x'^{(i+1)}+y'^{(i+1)})}{x'^{(i)}+y'^{(i)}}, \;\;\;\; y^{(i)} = \frac{x'^{(i)}(x'^{(i+1)}+y'^{(i+1)})}{x'^{(i)}+y'^{(i)}}.
\end{equation}
Here indices are taken modulo $n$, the number of horizontal wires in the picture.  These transformations appeared in \cite[Section 6]{LP} (see also \ref{ssec:whirlcurl} where this move is discussed in more detail).
\begin{figure}[h!]
    \begin{center}
    \input{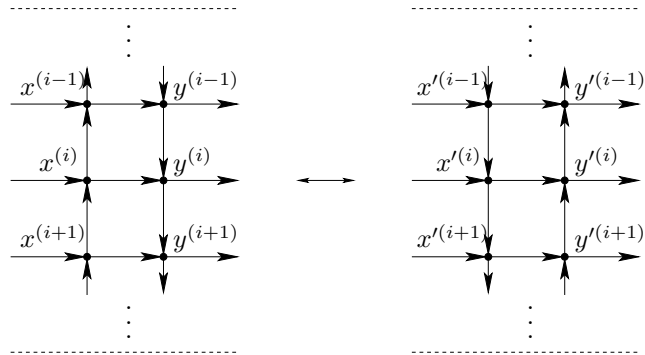}
    \end{center}
    \caption{The whirl-curl move on a cylinder.  The upper and lower dashed edges of the picture are identified.}
    \label{fig:wire12}
\end{figure}
\end{enumerate}

\begin{lemma}
Local transformations preserve the class of simple-crossing, vertex weighted oriented networks.
\end{lemma}

\begin{proof}
Trivial.
\end{proof}

\begin{thm} \label{thm:bm}
Local transformations preserve boundary and cycle measurements.
\end{thm}
\begin{proof}
We first discuss boundary measurements.  In all moves one checks directly that for any fixed pair of an enetering half-edge and an exiting half-edge, the accumulated weight over highway paths from through these half-edges is the same on both sides of the move.  For the (YB) move, this calculation is essentially \eqref{eq:L}, as in Example \ref{ex:disk}.  In the case of the (CR) move note that depending on orientation of the cycle, one either cannot enter it or cannot exit it. Thus the removal does not change any boundary or cycle measurements (recall that we do not consider cycle measurements with trivial homology). 

For cycle measurements, the proof is more complicated since the multiplicity of a highway cycle changes when one performs local transformations.

Let $G$ and $G'$ be the two subgraphs of a local transformation.  Thus a local transformation changes $\N$ to $\N'$ by replacing $G$ with $G'$.  It is easy to check that highway cycles completely contained in $G$ or $G'$ have the same contribution to cycle measurements.  Let $p$ be a highway cycle in $\N$.  Then $p$ may enter and exit $G$ multiple times.  Let us consider the set $S(\N)$ (resp. $S(\N')$) of highway cycles in $\N$ (resp. $\N'$) which agree with $p$ except for the parts of $p$ contained inside $G$ (resp. $G'$).  Note that any cycle $p' \in S(\N)$ will enter and leave $G$ via the same edges as $p$.  Let $m$ be the maximal multiplicity of a path $q$ in $S(\N)$.  Then it is easy to see the multiplicity of any path in $S(\N)$ or $S(\N')$ divides $m$.  Suppose $q$ is the $m$-multiple of a highway cycle $q'$.  Let $v \in q$ be a vertex in $\N - G = \N' - G'$.  Let $T(\N)$ (resp. $T(\N')$) be the set of highway paths $r$ in $\N$ (resp. $\N'$) which begin and end at $v$, and apart from changes within $G$ (resp. $G'$), are the same as traversing $q'$ $m$ times.  Thus each $r \in T(\N)$ is obtained from some highway cycle $p(r)$ in $S(\N)$ by ``cutting'' at $v$.  Each $p \in S(\N)$ corresponds to $m/\mult(p)$ paths in $T(\N)$.  Thus 
$$
 \sum_{p \in S(\N)} \frac{1}{\mult(p)}\wt(p) = \frac{1}{m} \sum_{r \in T(\N)} \wt(r).
$$
But it is clear from the calculation for boundary measurements that $\sum_{r \in T(\N)} \wt(r) = \sum_{r \in T(\N')} \wt(r)$, so that the contribution of $S(\N)$ and $S(\N')$ to cycle measurements are identical.
\end{proof}

\subsection{Wires} \label{ssec:wires}
A {\it wire} in $N$ is a path (not necessarily a highway path) in $N$ which goes straight through the crossing at every interior vertex.  Each wire is either a cycle, or starts and ends on the boundary.  Every edge of $N$ belongs to a unique wire.

A {\it whirl} is a wire cycle in $N$, such that it enters every vertex via the underway.  Alternatively, a whirl is a wire cycle for which other wires cross from the left.  A {\it curl} is a wire cycle in $N$ such that it enters every vertex via the highway.  Alternatively, a curl is a wire cycle for which other wires cross from the right.  A {\it whurl} (mix of a whirl and a curl) is a wire cycle with no self-intersections.

The Yang-Baxter move and the whirl-curl move are all compatible with wires in the following sense: each wire before the move corresponds to a wire after the move with the same endpoints (if any).  The crossing merging and crossing removal moves do not preserve wires in this manner, though the composition of a crossing merging and crossing removal (at the same crossing!) will achieve this.

\begin{figure}[h!]
    \begin{center}
    \input{wire28.pstex_t} 
    \end{center} 
    \caption{}
    \label{fig:wire28}
\end{figure}
We call an oriented network $N$ {\it reduced}, if wires do not self-intersect, and every pair of wires has the minimal number of intersections given their homology classes.  It is not always possible to change a weighted oriented network $\N$ to a reduced one using local moves. For example, it can be shown that a self-crossing wire cycle on a cylinder as shown in Figure \ref{fig:wire28} cannot be changed into a reduced network if the vertex weight $x \not = 0$.

\subsection{Torus action} \label{ssec:torus}
In addition to the five local transformations of Section \ref{ssec:localmoves}, there is a global transformation that preserves the boundary and cycle measurements.  This is an action of a torus $T = (K^*)^k$, where $K$ is a field which acts on vertex weights and $K^*$ is the multiplicative group of $K$.  For example if the vertex weights lie in the rational functions $\C(x_1,x_2,\ldots)$ in some number of indeterminates, we could choose $K = \C$. 

Let $\N$ be a vertex-weighted simple-crossing oriented network in $S$, and $N=N(\N)$ be the underlying unweighted oriented network.  Define a graph $H_N$ as follows.  The vertex set of $H_N$ is the set $\sp(N)$ of all snake paths and cycles in the network $N$ (see Section \ref{ssec:snake}).  At each interior vertex $v$ of $N$ the highway path goes from one snake path or cycle $p$ to another $q$. It is possible that $p=q$.  For each such vertex $v$, we add a directed edge $e_v$ from $p$ to $q$ in $H_N$.

The graph $H_N$ records which snake paths or cycles are traversed when one performs a highway walk in $N$.  Now let $V_N$ be the free $\Z$-module spanned by the edges of $H_N$, or equivalently by interior vertices of $N$. We introduce an inner product on $V_N$ by making the edges of $H_N$ orthonormal.  Call vertices of $H_N$ that correspond to snake paths {\it {terminal}}.  Define
$$
U = \{\mbox{closed walks in $H_N$}\} \cup \{\mbox{paths from terminal vertices to terminal vertices in $H_N$}\}
$$
and for $u \in U$, we let $a_u \in V_N$ denote the sum of all edges involved in $u$ with multiplicities.
Let $W_N \subset V_N$ be the submodule of $V_N$ given by 
$$
W_N = \{w \in V_N \mid \ip{w,a_u} = 0 \text{\;\; for all  } u \in U\}.
$$
Let $k$ be the dimension of $W_N$, and let $w_1, \ldots, w_k$ be a basis of $W_N$ over $\mathbb Z$.

Define the action of $T = (K^*)^k$ on $\N$ as follows.  For a vertex $v \in \N$, we let $x_v$ denote the vertex weight in $\N$.  Then for $\tau =(\tau_1, \ldots, \tau_k) \in (K^*)^k$, the 
assignment
\begin{equation}\label{eq:torus}
y_v = \prod_i \tau_i^{\langle e_v, w_i \rangle} x_v
\end{equation}
defines vertex weights of a network denoted $\tau \cdot \N$.

\begin{proposition}\label{prop:torus}
The definition \eqref{eq:torus} gives an action of $T$ that preserves all cycle and boundary measurements in $\N$.
\end{proposition}

\begin{proof}
The first part is clear.  Suppose $p$ is a highway path in $\N$.  Each highway crossing on $p$ corresponds to an edge in $H_N$.  It follows that one obtains a walk $u_p \in U$ from $p$,
 and in addition the weight $\wt(p)$ is equal to $\prod_{e_v \in u_p} x_v$.  But then 
$$
\prod_{e_v \in u_p} y_v = \prod_{e_v \in u_p}\prod_i \tau_i^{\langle e_v, w_i \rangle} x_v 
= \prod_i \tau_i^{\langle a_u,w_i \rangle} \wt(p) = \wt(p)
$$
Thus each highway path has the same weight in $\N$ as in $\tau \cdot \N$.  The same holds for highway cycles.
\end{proof}

\begin{example} \label{ex:tor}
Consider the network on a torus from Figure \ref{fig:wire6}, repeated in Figure \ref{fig:wire14}. Let $x_{11}$, $x_{12}$, $x_{13}$, $x_{21}$, $x_{22}$, $x_{23}$ and $x_{31}$, 
$x_{32}$, $x_{33}$ be the vertex weights. There are three snake cycles, the blue one, the green one and the magenta one. 
\begin{figure}[h!]
    \begin{center}
    \input{wire14.pstex_t}
    \end{center}
    \caption{}
    \label{fig:wire14}
\end{figure}
There are three edges $e_1$, $e_2$ and $e_3$ in $H_N$ forming a cycle of length three. The vector $e_1 + e_2+e_3 \in V_N$ spans ${\rm span}(a_u)_{u \in U}$.  The vectors 
$e_1 - e_2$, $e_2 - e_3$ span $W_N$ over $\mathbb Z$. The torus action in this case is given by $$(\tau_1,\tau_2) \colon (x_{11}, x_{12}, x_{13}, x_{21}, x_{22}, x_{23}, x_{31}, x_{32}, 
x_{33}) \mapsto$$ $$(\tau_2^{-1} x_{11}, \tau_1 x_{12}, \tau_1^{-1} \tau_2 x_{13}, \tau_1 x_{21}, \tau_1^{-1} \tau_2 x_{22}, \tau_2^{-1} x_{23}, \tau_1^{-1} \tau_2 x_{31}, \tau_2^{-1} 
x_{32}, \tau_1 x_{33}).$$ 
\end{example}

\begin{figure}[h!]
    \begin{center}
    \input{wire13.pstex_t}
    \end{center}
    \caption{}
    \label{fig:wire13}
\end{figure}
\begin{example}
Consider the network on a cylinder from Figure \ref{fig:wire6}, repeated in Figure \ref{fig:wire13}. Then the graph $H_N$ has four edges that we denote $e_1$, $e_2$, $e_3$ and $e_4$. 
There are no cycles in $H_N$, and the paths from terminal vertices to terminal vertices give us vectors $e_1+e_2+e_3+e_4$, $e_2+e_3+e_4$, $e_1+e_2+e_3$, $e_2+e_3$. The orthogonal 
complement is generated by $e_2-e_3$, which gives us action $$\tau \colon (x_{11}, x_{12}, x_{13}, x_{21}, x_{22}, x_{23}, x_{31}, x_{32}, x_{33}, x_{41}, x_{42}, x_{43}) \mapsto$$ 
$$(x_{11}, \tau x_{12}, \tau^{-1} x_{13}, \tau x_{21}, \tau^{-1} x_{22}, x_{23}, \tau x_{31}, \tau^{-1} x_{32}, x_{33}, x_{41}, \tau x_{42}, \tau^{-1} x_{43}).$$
\end{example}

\subsection{Monodromy action and the Frenkel-Moore relation}
Let $\N$ be a simple-crossing oriented network.  Sometimes a sequence of local transformations (the torus action of Section \ref{ssec:torus} is excluded) will change $\N$ into a network $\N'$ such that the underlying unweighted networks $N(\N)$ and $N(\N')$ are the same, but the vertex weights may not be.  We call this the {\it monodromy action} of the transformations on the vertex weights of $\N$.  This monodromy action essentially depends only on $N(\N)$.

For an unweighted oriented network $N$, we let $\C(N)$ denote the field of rational functions in indeterminates $x_v$, as $v$ varies over the vertices of $N$.  The local transformations give rise to a (possibly infinite) {\it monodromy group} $\M(N)$ acting on $\C(N)$ as birational transformations.  Some conjectures concerning $\M(N)$ are made in Section \ref{sec:conjectures}, and a special case is studied in Section \ref{ssec:monodromy}.

The {\it Frenkel-Moore relation} is an example of a sequence of local transformations which give a trivial monodromy action (but not obviously).

\begin{figure}[h!]
    \begin{center}
    \input{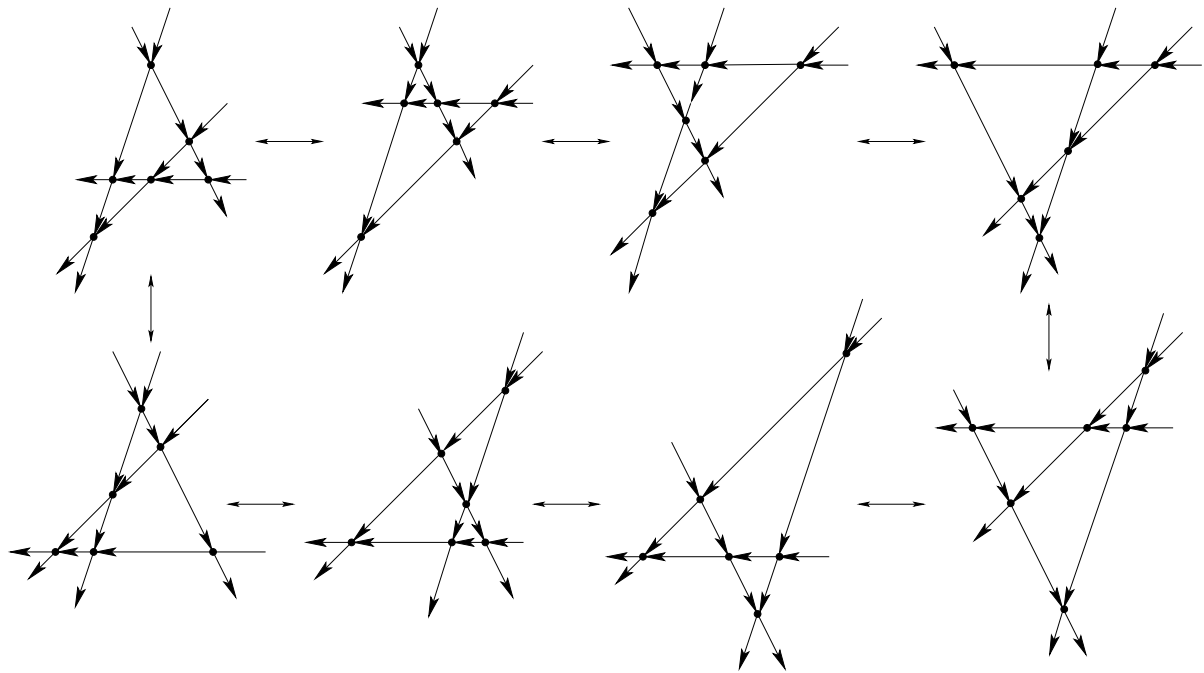}
    \end{center}
    \caption{The eight-cycle of (YB) moves; the monodromy action of this sequence of moves on the vertex weights is trivial, and is called the Frenkel-Moore relation.}
    \label{fig:wire15}
\end{figure}
Given a local part of a network that looks like four wires crossing, one can apply a sequence of Yang-Baxter moves as shown in Figure \ref{fig:wire15}, and come back to the original network.  From a direct computation, one obtains

\begin{prop} \label{prop:fm}
Performing the eight Yang-Baxter moves in the Frenkel-Moore relation does not change the vertex weights.
\end{prop}


We shall use the Frenkel-Moore relation in the following manner.  One can go from any diagram in the above cycle to any other diagram in two different ways.  Proposition \ref{prop:fm} tells us that the vertex weights in the result do not depend on which of the two ways is chosen.

\subsection{Underway paths and measurements}
Swapping left and right, one obtains the notion of {\it {underway paths}}, in which highway crossings are prohibited and the weight is computed using underway crossings. Thus one has {\it {underway boundary measurements}} and {\it {underway cycle measurements}}. By analogy with the highway paths case, one has

\begin{proposition} \label{prop:underbm}
Underway boundary and cycle measurements are invariant under local transformations.
\end{proposition}

Let $S^*$ be the mirror image of the surface $S$, and let $\N^*$ be the corresponding mirror image of the network $\N$.
\begin{prop}
 The boundary and cycle measurements for the network $(S^*, \N^*)$ coincide with the underway boundary and cycle measurements of the network $(S, \N)$.
\end{prop}


\section{Conjectures}\label{sec:conjectures}
A polyhedron in $\R^n$ is a subset specified by some affine linear inequalities or equivalently as the intersection of halfspaces.  A polyhedron is {\it pointed} if it does not 
contain a line.  We call a subset of $\Z^n$ a polyhedron if it is the intersection of $\Z^n$ with a polyhedron in $\R^n$.  

Let $H(u,v)$ denote the set
$\{[q \cup p^*]\} \subset H_1(S,\Z)$ of homology classes, where $p$ is a basepath from $u$ to $v$ and $q$ varies over all the highway paths $p$ between boundary vertices $u$ and $v$.  
Note that changing the basepath $p$ amounts to translating the set $H(u,v)$ by a fixed vector.   Thus Lemma \ref{lem:snake} says that $H(u,v)$ is contained in the polyhedron specified 
by the affine inequalities $\ip{q,r} \geq \ip{p,r}$ for each snake path or cycle $r$.  This is strengthened by

\begin{conjecture}
$H(u,v)$ is a pointed polyhedron.
\end{conjecture}

The following conjecture sharpens Proposition \ref{prop:rational}.

\begin{conjecture}\label{conj:rat}
For an appropriate choice of basis of $H_1(S)$, the generating function $M_{u,v}$ of Section \ref{ssec:rat} is the product of a (Laurent) monomial and a rational function in the $t_i$.
\end{conjecture}

We will provide some evidence for the following conjecture in Theorem \ref{thm:SaSb}.

\begin{conjecture}
When an oriented network $N$ is {\it reduced}, the monodromy action on $N$ is the action of a discrete group.  
\end{conjecture}
Here, discrete-ness should be interpreted in terms of the topology of the field containing the vertex weights.  This topology should be such that the local transformations are continuous.

We believe the following problems have an affirmative answer for the cylinder and the torus, but are not sure in the case of higher genus surfaces.

\begin{problem}\label{pr:1}
 Determine for which surfaces the following is true. Any two assignments of weights to vertices of an unweighted network that yield the same boundary and cycle measurements, can be connected by monodromy and torus actions.
\end{problem}

\begin{problem} \label{pr:2}
  Determine for which surfaces the following is true. For any fixed underlying unweighted network, the boundary and cycle measurements generate the field of rational invariants of monodromy and torus actions.
\end{problem}

For certain networks on the cylinder, Problems \ref{pr:1} and \ref{pr:2} are solved in Theorems \ref{thm:ntol} and \ref{thm:LP4}.

\begin{remark}
The torus action is necessary for Problems \ref{pr:1} and \ref{pr:2}: the statements would be false with only the monodromy action considered.  For example, it can be shown that in Example \ref{ex:tor} all measurements have degree divisible by three.  On the other hand, the local transformations preserve the sum of all weights, which has degree one. Thus there are invariants of the monodromy action that are not rational functions in the measurements. 
\end{remark}

\section{The sphere and the disk}\label{sec:spheredisk}
\subsection{The sphere}\label{ssec:sphere}
Suppose $S$ is a sphere.  Then there are no boundary measurements, and the only cycle measurement is that of the trivial homology class, for which many of our earlier results do not apply.

\begin{proposition}
Let $\N$ be an oriented simple-crossing network on the sphere.  Then $\N$ can be transformed to the empty network using the local transformations of Section \ref{ssec:localmoves}.
\end{proposition}
\begin{proof}
Suppose $\N$ is non-empty.  Find a simple oriented cycle $C$ in $\N$, separating $S$ into two components which are disks.  Take one of these disks $D$.  If there is a part of $\N$ in the interior of $D$, then we can either find a smaller simple oriented cycle strictly inside $D$, or we can find an oriented path inside $D$ which joins two vertices on $C$.  In either case, we can find a smaller simple oriented cycle.  Repeating, we eventually obtain a simple oriented cycle $C'$ which bounds a disk $D'$ not containing any part of $\N$.  We can now remove $C'$ using the cycle removal (CR) move.  But $\N$ is finite, so after repeating this finitely many times, we are left with the empty network.
\end{proof}

\subsection{Postnikov's plabic graphs and the disk}\label{ssec:disk}
Suppose $S$ is a disk and $\N$ is an oriented simple-crossing network on $S$.  The network $G(\N)$ in the proof of Proposition \ref{prop:rational} is a trivalent edge-weighted network, such that every interior vertex either has indegree one, or outdegree one. 

In \cite{Pos}, Postnikov studied a class of oriented networks in a disk.  The network $G(\N)$ is, in the terminology of \cite{Pos} a ``trivalent perfect network''.  It gives rise to a plabic network $P(\N)$ in the sense of Postnikov as follows: each outdegree one vertex is colored black and each indegree one vertex is colored white, and the orientations are removed.  Following Postnikov \cite[Section 11]{Pos} the edge weights of $G(\N)$ give rise to face weights of $P(\N)$. 

\begin{prop}
Our boundary measurements $M^{[p]}$ of $\N$ agree with Postnikov's boundary measurements for $G(\N)$ (or $P(\N)$).
\end{prop}
\begin{proof}
Since the homology of the disk is trivial, such measurements $M^{[p]}$ only depend on the source and sink vertices.  By Theorem \ref{thm:fin}, there are only finitely many paths from each source to each sink.  In particular, $G(\N)$ does not contain any oriented cycles which can be traversed by such paths.  In this case, both ours and Postnikov's boundary measurements agree with the usual weight generating functions of paths in edge-weighted graphs.  (One of the main innovations of \cite{Pos} is to tackle the infinite sums which occur when oriented cycles are present.)
\end{proof}

In \cite[Section 12]{Pos}, Postnikov describes a number of local transformations on plabic networks.  Only four of our local transformations: Yang-Baxter move, cycle removal, crossing merging, and crossing removal can occur for networks in a disk.  Since Postnikov's weights are variables (and in particular cannot be 0), there is no analogue of our crossing removal.

We assume familiarity with Postnikov's transformations of plabic graphs.
\begin{prop}
Under the transformation $\N \mapsto P(\N)$
\begin{enumerate}
\item
Our Yang-Baxter move becomes Postnikov's square move composed with a number of edge contraction/uncontractions.
\item
Our cycle removal becomes a composition of edge contractions, loop removal, leaf reductions, and middle vertex removals.
\item
Our crossing removal becomes a composition of an edge contraction, an edge uncontraction, and a parallel edge reduction.
\end{enumerate}
\end{prop}
\begin{proof}
Direct computation using the construction of $G(\N)$ (Figure \ref{fig:wire7}). 
\begin{figure}[h!]
    \begin{center}
    \input{wire29.pstex_t} 
    \end{center} 
    \caption{}
    \label{fig:wire29}
\end{figure}
Figure \ref{fig:wire29} shows a sequence of transformations realizing the Yang-Baxter move by Postnikov's square move.
\end{proof}

\part{Cylindrical networks and total positivity in loop groups}

\section{Whurl relations via local moves} \label{sec:whurl}
Recall the definition of whirl, curl, and whurl cycles from Section \ref{ssec:wires}.
\subsection{The whurl relations}\label{ssec:whurl}
Let $\N$ be an oriented simple-crossing network on a surface $S$.  We assume that on a local portion $S'$ of $S$ that looks like the cylinder, $\N' = \N \cap S'$ contains two homologically non-trivial wire cycles $C$, $C'$ which are whurls.  Furthermore we will assume that the other wires pass through $S'$ intersecting only $C$ and $C'$.  We shall call these other wires the horizontal wires, as pictured in Figure \ref{fig:wire18}.  We shall generally view the vertex weights of $\N$ as indeterminates, allowing us to ignore issues with denominators being 0 in the (YB) move.

We will assume that $\N'$ does not contain oriented cycles bounding faces.  Thus we have either 
\begin{enumerate}
\item
$C$ and $C'$ are oriented in the same direction, or
\item
one of $C$ and $C'$ is a whirl, and the other is a curl.
\end{enumerate}

In such a situation our local transformations allow us to swap the wires $C$ and $C'$.  In situation (1), we call this relation the {\it whurl-relation}, which specializes to the {\it whirl relation} or {\it curl relation} if $C$ and $C'$ are both whirls or both curls.  In situation (2), this is the whirl-curl relation (WC) of Section \ref{ssec:localmoves}.

For the remainder of this subsection, we assume that we are in situation (1).  We shall further assume that the network is rotated so that the whurls $C$ and $C'$ are both oriented up as in Figure \ref{fig:wire18}, and that $C$ is on the left of $C'$.  Number the horizontal wires cyclically from $1$ to $n$.  Assume the $i$-th wire intersects $C$ and $C'$ with weights $x^{(i)}$ and $y^{(i)}$ respectively, for $i = 1, \ldots, n$.  

To describe the transformation, introduce the variables $z = (z^{(1)}, \ldots, z^{(n)})$ and $t = (t^{(1)}, \ldots, t^{(n)})$ as follows:
$$
z^{(i)} = \begin{cases} 
x^{(i)} & \mbox{if the $i$-th horizontal wire is oriented to the right,}\\
y^{(i)}& \mbox{if the $i$-th horizontal wire is oriented to the left;}
\end{cases}
$$
$$
t^{(i)} = \begin{cases} 
y^{(i-1)} & \mbox{if the $(i-1)$-st horizontal wire is oriented to the right,}\\
x^{(i-1)}& \mbox{if the $(i-1)$-st horizontal wire is oriented to the left.}
\end{cases}
$$
where we take indices modulo $n$.  Let 
$$\epsilon_i = \begin{cases} 
1 & \mbox{if the $i$-th horizontal wire is oriented to the right,}\\
0 & \mbox{if the $i$-th horizontal wire is oriented to the left.}
\end{cases}
$$ 
Finally, define 
\begin{equation}\label{eq:kappa}
\k_{r}(x,y) = \sum_{s=0}^{n-1} (\prod_{t=1}^s y^{(r+t)} \prod_{t=s+1}^{n-1} x^{(r+t)}).
\end{equation}

Then the new values of $x$-s and $y$-s after the whurl transformation are defined to be 
\begin{equation} \label{eq:whurl}
x'^{(i)} = \frac{y^{(i)}  \k_{i+\epsilon}(z,t)}{\k_{i+1-\epsilon}(z,t)} \text{\;\;\;\; and \;\;\;\;} y'^{(i)} = \frac{x^{(i)} \k_{i+1-\epsilon}(z,t)}{\k_{i+\epsilon}(z,t)}.
\end{equation}

\begin{figure}[h!]
    \begin{center}
    \input{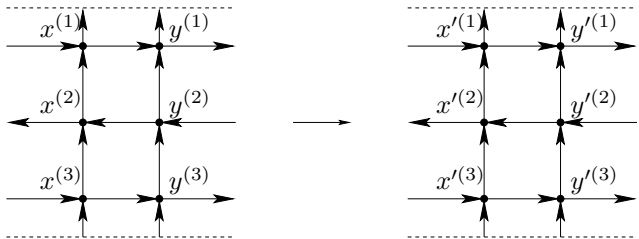}
    \end{center}
    \caption{An example of whurl relation.}
    \label{fig:wire18}
\end{figure}
The whurl relation for Figure \ref{fig:wire18} is given by
$$x'^{(1)}=y^{(1)} \frac{x^{(1)}x^{(2)}+x^{(1)}x^{(3)}+x^{(2)}y^{(3)}}{y^{(2)}x^{(3)}+y^{(1)}x^{(3)}+y^{(1)}x^{(2)}} \;\;\; x'^{(2)}=y^{(2)} \frac{x^{(1)}x^{(2)}+x^{(1)}x^{(3)}+x^{(2)}y^{(3)}}{x^{(1)}y^{(2)}+y^{(1)}y^{(3)}+y^{(2)}y^{(3)}}$$ $$ x'^{(3)}=y^{(3)} \frac{y^{(1)}x^{(3)}+y^{(1)}x^{(2)}+y^{(2)}x^{(3)}}{x^{(1)}y^{(2)}+y^{(1)}y^{(3)}+y^{(2)}y^{(3)}}$$
$$y'^{(1)}=x^{(1)} \frac{y^{(2)}x^{(3)}+y^{(1)}x^{(3)}+y^{(1)}x^{(2)}}{x^{(1)}x^{(2)}+x^{(1)}x^{(3)}+x^{(2)}y^{(3)}} \;\;\; y'^{(2)}=x^{(2)} \frac{x^{(1)}y^{(2)}+y^{(1)}y^{(3)}+y^{(2)}y^{(3)}}{x^{(1)}x^{(2)}+x^{(1)}x^{(3)}+x^{(2)}y^{(3)}}$$
$$  y'^{(3)}=x^{(3)} \frac{x^{(1)}y^{(2)}+y^{(1)}y^{(3)}+y^{(2)}y^{(3)}}{y^{(1)}x^{(3)}+y^{(1)}x^{(2)}+y^{(2)}x^{(3)}}$$
 In this example, one has $z=(x^{(1)}, y^{(2)}, x^{(3)})$, $t=(y^{(3)}, y^{(1)}, x^{(2)})$ so that $\k_1(z,t) = y^{(2)} x^{(3)} + y^{(1)} x^{(3)} + y^{(1)} x^{(2)}$, and so on.

\subsection{Whirl-curl relation}\label{ssec:whirlcurl}
\begin{figure}[h!]
    \begin{center}
    \input{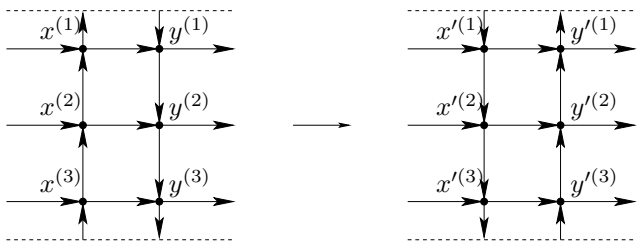}
    \end{center}
    \caption{An example of the whirl-curl relation.}
    \label{fig:wire19}
\end{figure}
The whirl-curl relation for Figure \ref{fig:wire19}, stated in Section \ref{ssec:localmoves}, is given by
$$x'^{(1)} = y^{(1)} \frac{x^{(3)}+y^{(3)}}{x^{(1)}+y^{(1)}} \;\;\; x'^{(2)} = y^{(2)} \frac{x^{(1)}+y^{(1)}}{x^{(2)}+y^{(2)}} \;\;\; x'^{(3)} = y^{(3)} \frac{x^{(2)}+y^{(2)}}{x^{(3)}+y^{(3)}}$$
$$y'^{(1)} = x^{(1)} \frac{x^{(3)}+y^{(3)}}{x^{(1)}+y^{(1)}} \;\;\; y'^{(2)} = x^{(2)} \frac{x^{(1)}+y^{(1)}}{x^{(2)}+y^{(2)}} \;\;\; y'^{(3)} = x^{(3)} \frac{x^{(2)}+y^{(2)}}{x^{(3)}+y^{(3)}}$$
%

\subsection{Realization by Yang-Baxter and crossing moves} \label{sec:R}
Now we describe a sequence of local operations that globally produces a whurl or a whirl-curl (WC) move.  In the whurl case, we begin by applying (XR) and (XM) moves to create a ``ripple'' with weights $p$ and $-p$ as shown in Figure \ref{fig:wire20}.  In the whirl-curl case, we also create such a ripple, though this does not follow from (XR) and (XM), since the wires are oriented in opposite directions. (The creation of this ripple in the whirl-curl case {\it does} change the boundary and cycle measurements.)

Once the ripple has been created, one can start pushing one of the crossings of the ripple through the horizontal wires using the (YB) move.  One repeats this until the crossing revolves once around the cylinder.  Note that in the whirl-curl situation, only one of the two crossings can be moved, since moving the other would require performing a (YB) move on an oriented three-cycle, which is not defined. 
\begin{figure}[h!]
    \begin{center}
    \input{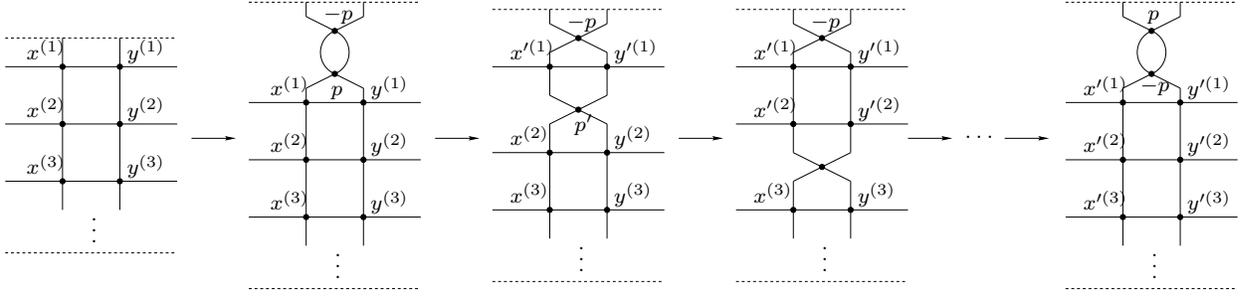}
    \end{center}
    \caption{For a unique choice of the weight $p$, the weight that comes out on the other side after passing through all horizontal wires is also $p$; the resulting transformation of $x$ and $y$ is exactly the whurl move, or the whirl-curl move.}
    \label{fig:wire20}
\end{figure}


\begin{lemma} \label{lem:unp}
 There is at most one value of $p$ for which the weight of the crossing once it has been pushed through all wires is equal to the original value $p$.
\end{lemma}

\begin{proof}
 The proof for the whurl move and for the whirl-curl move is different. 

In the whurl case, it is easy to argue by induction from the Yang-Baxter relation that the weight of the crossing as it travels through the wires always has the form $\frac{Ap}{Bp+C}$, where $A, B, C$ do not depend on $p$. Thus at the end one obtains a linear equation in $p$, which has a unique solution or no solutions at all. 

In the whirl-curl case, we observe directly that the weight to come out on the other side is always $x^{(n)} + y^{(n)}$. Thus, $p=x^{(n)} + y^{(n)}$ is obviously the unique solution.
\end{proof}

In the following we assume the wires have been numbered as in Figure \ref{fig:wire20}.  Namely, the ripple is created above the $1$-st wire.

\begin{theorem} \label{thm:whloc}
Suppose we are in the whurl move case.  If we set $p = \frac{\prod_i t^{(i)} - \prod_i z^{(i)}}{\k_1(z,t)}$ and perform the sequence of (YB) moves in Figure \ref{fig:wire20}, then at the moment when the crossing is between the $j$-th and $(j+1)$-st wires, the value of the crossing weight is equal to $\frac{\prod_i t^{(i)} - \prod_i z^{(i)}}{\k_{j+1}(z,t)}$, and the resulting values of $x'^{(i)}$ and $y'^{(i)}$ are as in the whurl transformation \eqref{eq:whurl}.
\end{theorem}

\begin{proof}
It is enough to consider one (YB) move. Note that the edge is always oriented from the vertex weighted $z^{(j)}$ to the vertex weighted $t^{(j+1)}$. One obtains 
\begin{align*}z'^{(j)} &= z^{(j)} + \frac{\prod_i t^{(i)} - \prod_i z^{(i)}}{\k_{j}(z,t)}\\ 
&= \frac{\prod_i t^{(i)} - \prod_i z^{(i)} + z^{(j)} \sum_{s=0}^{n-1} (\prod_{r=1}^s t^{(j+r)} \prod_{r=s+1}^{n-1} z^{(j+r)})}{\k_{j}(z,t)}\\
&=\frac{\prod_i t^{(i)} + \sum_{s=1}^{n-1} (\prod_{r=1}^s t^{(j+r)} \prod_{r=s+1}^{n} z^{(j+r)})}{\k_{j}(z,t)}\\ &= t^{(j+1)} \frac{\sum_{s=0}^{n-1} (\prod_{r=1}^s t^{(j+1+r)} \prod_{r=s+1}^{n-1} z^{(j+1+r)})}{\k_{j}(z,t)} = t^{(j+1)} \frac{\k_{j+1}(z,t)}{\k_{j}(z,t)},
\end{align*}
\begin{align*}
t'^{(j+1)} &= \frac{t^{(j+1)}z^{(j)}}{z^{(j)} + \frac{\prod_i t^{(i)} - \prod_i z^{(i)}}{\k_{j}(z,t)}} = \frac{t^{(j+1)}z^{(j)}}{t^{(j+1)} \frac{\k_{j+1}(z,t)}{\k_{j}(z,t)}} = z^{(j)} \frac{\k_{j}(z,t)}{\k_{j+1}(z,t)},
\end{align*}
\begin{align*}p' &= \frac{t^{(j+1)} \frac{\prod_i t^{(i)} - \prod_i z^{(i)}} {\k_{j}(z,t)}} {z^{(j)} + \frac{\prod_i t^{(i)} - \prod_i z^{(i)}}{\k_{j}(z,t)}} = \frac{\prod_i t^{(i)} - \prod_i z^{(i)}}{\k_{j+1}(z,t)}.
\end{align*}
It is easily seen that these formulae for $z'^{(j)}$ and $t'^{(j+1)}$ are equivalent to the formulae for $x'^{(j)}$ and $y'^{(j)}$ in \eqref{eq:whurl}.
\end{proof}

\begin{theorem} \label{thm:wcloc}
Suppose we are in the whirl-curl move case, and that the whirl wire is to the left of the curl wire as in the left diagram of Figure \ref{fig:wire19}.  If we set $p = x^{(n)}+y^{(n)}$ and perform the operation in Figure \ref{fig:wire20}, then at the moment when the crossing is between $j$-th and $(j+1)$-st wires, the value of the crossing weight is equal to $x^{(j)}+y^{(j)}$, and the resulting values of $x'^{(i)}$ and $y'^{(i)}$ are as in the whirl-curl transformation \eqref{eq:whirlcurl}. 

A similar statement holds when the whirl wire is to the right of the curl wire as in the right diagram of Figure \ref{fig:wire19}.  In this case the crossing is pushed upwards (opposite to Figure \ref{fig:wire20}) and one obtains \eqref{eq:whirlcurl2}.
\end{theorem}
\begin{proof}
The statement follows by direct computation.
\end{proof}

\subsection{Invariance of measurements, and inverses}
As was commented in Section \ref{ssec:localmoves}, the fact that the (WC) move preserves the boundary and cycle measurements follows from a simple direct calculation.  The same is true for whurl moves.  This is less obvious, but follows from Theorem \ref{thm:whloc} and Theorem \ref{thm:bm}. 

\begin{corollary}\label{cor:whurlbm}
The whurl transformation preserves the boundary and cycle measurements.
\end{corollary}
%

\begin{theorem} \label{thm:inv}
The whurl transformation is an involution. The two whirl-curl transformations \eqref{eq:whirlcurl} and \eqref{eq:whirlcurl2} are inverses of each other. 
\end{theorem}

\begin{proof}
In the case of the whurl relation one argues as follows.  If one applies the (YB) move to successively push two crossings, with weights $q$ and $-q$ negative of each other, through a wire $W$, then there is no net effect on the weights of the two crossings along $W$.  Also, the weights on the two crossings pushed through remain negatives of each other. Thus if one pushes through $p$ to perform a whurl move, and then a $-p$, the latter will perform a whurl move which will undo the first one. 

In the case of the whirl-curl relations, this follows from invertibility of the (YB) move, and the fact that \eqref{eq:whirlcurl} is obtained by pushing a crossing downwards, while \eqref{eq:whirlcurl2} is obtained by pushing a crossing upwards.
\end{proof}

\subsection{Braid relation for the whurl move, and whirl-curl move}
For this subsection, we assume that we have a local part $\N' = \N \cap S'$ of a network on a part $S'$ of the surface which is a cylinder, such that $\N'$ consists of three whurls crossed by a family of horizontal wires.  As before, we assume that the horizontal wires do not cross each other and that the are no oriented cycles bounding a face. 

Let us denote by $R$ an application of the whurl move, or whirl-curl move, to a pair of adjacent whurls.  We write $R_{12}$ or $R_{23}$ depending on the position of the whurls from the left.  For example, if the first and second whurls are oriented in opposite directions, then $R_{12}$ will apply a whirl-curl (WC) transformation.  The following theorem is a generalization of Theorem \cite[Theorem 6.3]{LP}.

\begin{theorem}\label{thm:braid}
 The transformation $R$ satisfies the following braid relation: 
\begin{equation}\label{eq:braid}
(R_{12} \otimes 1) \circ (1 \otimes R_{23}) \circ (R_{12} \otimes 1) = (1 \otimes R_{23}) \circ (R_{12} \otimes 1) \circ (1 \otimes R_{23}).
\end{equation}
\end{theorem}
 
\begin{proof}
Assume to begin with that our three whurls are not oriented down-up-down, or up-down-up.  In the space between the last and the first horizontal wires , create a triple ripple involving all three whurls, as shown in the left diagram of Figure \ref{fig:wire21}.
\begin{figure}[h!]
    \begin{center}
    \input{wire21.pstex_t}
    \end{center}
    \caption{}
    \label{fig:wire21}
\end{figure}
Assign weights $p$, $q$, $r$ so that when we push these crossings through downwards, they come out on the other side unchanged as in Lemma \ref{lem:unp}.  We assume here that they can be pushed through, since if they cannot, one can push their negatives through in the other (upwards) direction; here we are using the assumption that the orientation of whurls is not up-down-up or down-up-down. According to Lemma \ref{lem:unp} and Theorems \ref{thm:whloc} and \ref{thm:wcloc} the weights exist, are unique and are exactly the weights that would perform the transformation $(R_{12} \otimes 1) \circ (1 \otimes R_{23}) \circ (R_{12} \otimes 1)$.  Now, apply a sequence of Yang-Baxter moves as shown in Figure \ref{fig:wire21} to obtain weights $p'$, $q'$ and $r'$, and their negatives on the other side. We claim that if these new weights are pushed around the cylinder, they come out the same at the other end. This follows from the Frenkel-Moore relation (Proposition \ref{prop:fm}), which says that applying a Yang-Baxter move before or after pushing through a line produces the same result. Now if we push the weights $p$, $q$ and $r$ through and apply the Yang-Baxter moves, we obviously get again the weights $p'$, $q'$ and $r'$. Thus by Lemma \ref{lem:unp} and Theorems \ref{thm:whloc} and \ref{thm:wcloc}, while pushing $p'$, $q'$ and $r'$ through we are applying $(1 \otimes R_{23}) \circ (R_{12} \otimes 1) \circ (1 \otimes R_{23})$. Since, by the Frenkel-Moore relation (Proposition \ref{prop:fm}) the two results are the same, the claim of the theorem follows.

Now we claim that the remaining up-down-up or down-up-down cases follow.  Indeed, suppose that not all three whurls have the same orientation; it follows that all horizontal wires go in the same direction.  We may assume that we have two whirls and one curl.  To keep track of the orientations of the wire cycles, we for example write $WCW$ to mean that the first wire cycle is a whirl, the second is a curl and the last is a whirl again. When we perform the moves of \eqref{eq:braid}, we are taking two different paths to the opposite vertex in the hexagon of Figure \ref{fig:wire22}. 
\begin{figure}[h!]
    \begin{center}
    \input{wire22.pstex_t}
    \end{center}
    \caption{}
    \label{fig:wire22}
\end{figure}
We have already shown that if we start at a vertex $WWC$ or $CWW$, the two paths to the opposite vertex give the same result.  Since all the edges in the hexagon are involutions it follows that the same is true even if we started at the $WCW$ vertices of the hexagon.
\end{proof}

\section{Total positivity in cylindric networks}\label{sec:TP}
Let $\N$ be a simple-crossing, vertex-weighted, oriented network on the cylinder. Let $\tilde \N$ denote the universal cover of $\N$ inside a planar strip. Thus $\tilde \N$ is an infinite periodic network (see Figure \ref{fig:wire17}).

\begin{proposition}
Using the cycle-removal (CR) move, we can reduce $\N$ to a network $\N'$ which does not contain any contractible oriented cycles.
\end{proposition}
\begin{proof}
Assume there is a contractible oriented cycle in the network. Any such cycle is still a contractible cycle when lifted to the universal cover $\tilde \N$. If it is self-crossing, one can clearly take a part of it 
that also forms a contractible cycle on the universal cover, but has less self-crossings. Thus we can assume that there exists a simple contractible cycle on the universal cover. 
Order such cycles by containment, and take a minimal one $\tilde p$. Let us consider the disk it bounds in the universal cover.  If there is any part of $\tilde \N$ inside this disk, split it into wires using the usual rule 
(always go through at a crossing). The wires we obtain either are closed wires inside the disk, or go from its boundary to its boundary. Either way we get a smaller contractible 
oriented cycle -- contradiction. Thus the cycle $\tilde p$ has no part of $\tilde \N$ inside. This means that at every vertex it turns, and furthermore it either always turns left, or always 
turns right.  This in turn implies that the image $p$ of $\tilde p$ in $\N$ does not self-intersect and has no part of $\N$ inside, 
and therefore it can be removed using the (CR) move. Since the the network cannot get smaller forever, the statement follows.
\end{proof}
We shall henceforth assume that $\N$ contains no contractible oriented cycles.

We wish to establish an analogue of the Linstr\"om lemma \cite{Li} in such a situation. In \cite{LP} we established such an analogue for a general class of edge-weighted acyclic networks, but with the requirement that all sources are on one component of the boundary, while all sinks are on the other component of the boundary.

\begin{figure}[h!]
    \begin{center}
    \input{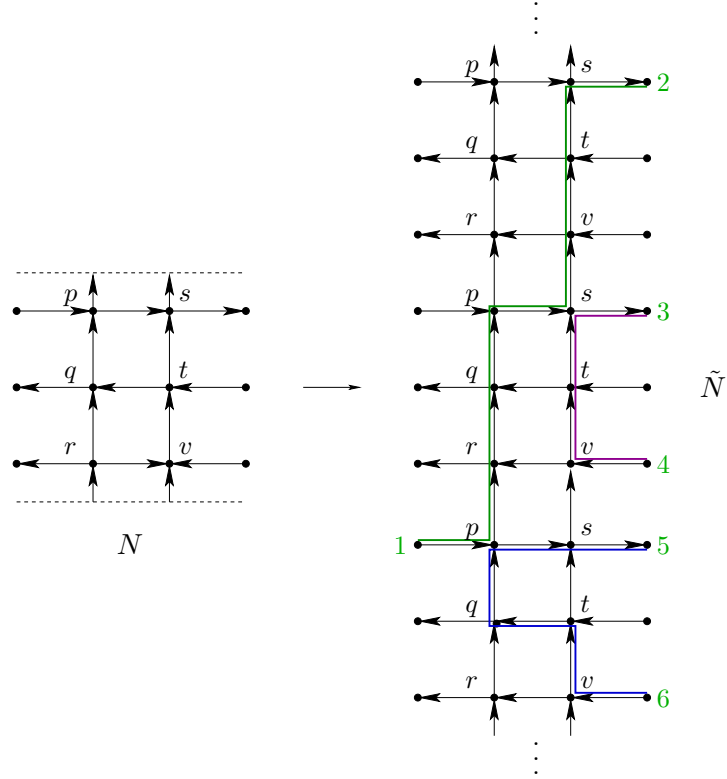}
    \end{center}
    \caption{A network on a cylinder and its universal cover; a non-crossing family of paths on it.}
    \label{fig:wire17}
\end{figure}

\begin{lemma}
If $\N$ has no contractible cycles then $\tilde \N$ is acyclic.
\end{lemma}
\begin{proof}
Any cycle in $\tilde \N$ would give a contractible cycle in $\N$.
\end{proof}

Let $\{v_i\}_i$ and $\{w_j\}_j$ be the sources and the sinks of the original cylindric network, and let $v_i^{k}$ and $w_j^{k}$ be their preimages in $\tilde \N$, $k \in \mathbb Z$.  Let $q(v_i^k, w_j^{l+k})$ denote a path from $v_i^k$ to $w_j^{l+k}$.  All such paths are homologous.

\begin{lemma} \label{lem:cov}
For a source $v_i$, a sink $w_j$, and a a highway path $p$ from $v_i$ to $w_j$, the boundary measurement $M^{[p]}(\N)$ coincides with the measurement $M^{[q(v_i^k, w_j^{l+k})]}(\tilde \N)$ for a particular value of $l$ and for any $k$.  Conversely, any boundary measurement of $\tilde \N$ is equal to some boundary measurement of $\N$.
\end{lemma}

\begin{proof}
Lift $p$ to the universal cover $\tilde \N$. Assume such a lift $\tilde p$ starts at $v_i^k$ and ends at $w_j^{k+l}$. It is easy to see that the homology type of $\tilde p$ depends only on $l$, giving the first statement.  The second statement follows easily from this.
\end{proof}

Let $I$ and $J$ be finite sets of sources and sinks in $\tilde \N$ of equal cardinality $|I|=k=|J|$. Let us create the corresponding {\it {boundary measurement matrix}} as follows, following the construction in \cite{Pos}. Number the elements of $I \cup J$ clockwise with $1,2,\ldots,|I|+|J|$, starting at any point. Let $i_1 < \ldots < i_k$ be the numbers assigned to sources, and $j_1 < \ldots < j_k$ be the numbers assigned to sinks. Define a $k \times k$ matrix $A_{I,J} = (a_{tr})_{t,r=1}^k$ by $$a_{tr} = (-1)^{s(i_t, j_r)} M^{[q(i_t, j_r)]},$$ where $s(i_t,j_r)$  denotes the number of elements of $I$ strictly between $i_t$ and $j_r$ in the clockwise order. 

\begin{example}\label{ex:matchings}
 Take the three sources and the three sinks numbered $1$ through $6$ in Figure \ref{fig:wire17}. Thus $I=\{1,4,6\}$ and $J = \{2,3,5\}$. One computes the boundary measurement matrix in this case to be  
$$A_{I,J} = \left(
\begin{matrix}
(pv+pq+rq)(vt+vs+qs) & pvt+pvs+pqs+rqs & -ps \\
vt+vs+qs & t+s & 0 \\
-(v+q)(vt+vs+qs) & -(vt+vs+qs) & t+s 
\end{matrix}
\right).
$$

\end{example}

We call a family of highway paths {\it {non-crossing}} if no two paths share a common edge. Denote by $F(I,J)$ all families of $k$ non-crossing highway paths from sources in $I$ to sinks in $J$ such that all elements of $I$ and $J$ are used. For a collection of paths $P=(p_1, \ldots, p_k)$ let $\wt(P) = \prod_{j=1}^k \wt(p_j)$.

\begin{theorem} \label{thm:tpd}
Let $\N$ be a network on the cylinder, and let $I$ and $J$ be finite sets of sources and sinks of the universal cover $\tilde \N$ that have equal cardinality.  Then $\det(A_{I,J}) = \sum_{P \in F(I,J)} \wt(P)$. Hence $\det(A_{I,J}) \geq 0$ if all vertex weights of $\N$ are nonnegative real numbers. 
\end{theorem}

This theorem is an analogue of \cite[Theorem 1.1]{Tal}, where it is proved for a large class of edge-weighted networks. One can deduce Theorem \ref{thm:tpd} from \cite[Theorem 1.1]{Tal} by constructing an edge-weighted network that would have the same measurements as $\tilde N$.  We give a direct proof here, since for the class of networks we consider the argument is not hard. 

\begin{proof}
Call two pairs $(i,j)$ and $(i',j')$ {\it {crossing}} if $(i-i')(i-j')(j-i')(j-j')<0$. For a matching $\sigma$ between elements of $I$ and $J$ let $\xing(\sigma)$ be the number of crossing pairs in $\sigma$. For a family $P$ of highway paths from elements of $I$ to elements of $J$ denote $\sigma(P)$ the matching between $I$ and $J$ paths in $P$ form. We first claim that $$\det(A_{I,J}) = \sum_P (-1)^{\xing(\sigma(P))} \wt(P).$$ Indeed, it suffices to check that the sign $\sgn(\sigma)$ the matching $\sigma$ gets in the determinant $\det(A_{I,J})$ is equal to $\xing(\sigma) \prod_{(i,j)\in \sigma} (-1)^{s(i,j)}$. This is easily verified, see \cite[Proposition 2.12]{Tal}.

Now we present a sign-reversing weight-preserving involution on crossing families of highway paths from $I$ to $J$. Assume that $r$ is smallest such that the path $p_r: i_r \to j_t$ intersects one of the other paths in the family. Among all such edges of intersection on $p_r$ choose the one, denoted $e$, that comes earliest. Among all paths that intersect $p_r$ at $e$ choose $p_s: i_s \to j_q$ with smallest possible $s$. Create $p_r'$ and $p_s'$ from $p_r$ and $p_s$ by changing their parts that come after $e$. Then it is easy to see that the new family of paths has the same weight but opposite sign. It is also easy to see this operation is an involution. Thus we can cancel out all terms in $\sum_P (-1)^{\xing(\sigma(P))} \wt(P)$ that have crossings. The remaining terms all have $\xing(\sigma(P)) = 0$, and the statement of the theorem follows.
\end{proof}

\begin{example}
 The determinant of the matrix in Example \ref{ex:matchings} is $$pqvt^3+pq^2st^2+rqvt^3+rq^2st^2+rq^2s^2t+pvqst^2+2rqvt^2s+rqvts^2.$$ A non-crossing family of highway paths that contributes the term $rqvt^2s$ of the determinant is shown in Figure \ref{fig:wire17}.
\end{example}

\begin{conjecture}\label{conj:tpd}
Theorem \ref{thm:tpd} characterizes the measurements of cylindric networks with nonnegative real weights.  In other words, if one is given a collection of nonnegative boundary measurements satisfying the inequalities of Theorem \ref{thm:tpd}, then these measurements can be represented by a network $\N$ on the cylinder with nonnegative real weights.
\end{conjecture}

\section{Whirl-curl-Bruhat cells and total positivity}\label{sec:whirlcurl}
Our background reference for this section is \cite{LP}.
\subsection{Unipotent loop group}
Let $\tS_n$ denote the affine symmetric group, the Coxeter group generated by involutions $\{s_i \mid i \in \Z/n\Z\}$.  For $n > 2$ the relations are $s_i s_{i+1} s_i = s_{i+1} s_i s_{i+1}$ and $s_i s_j = s_j s_i$ for $|i-j| \geq 2$.  For $n = 2$, there are no relations.

Let $G= GL_n(\R((t)))$ denote the (real) formal loop group, consisting of non-singular $n\times n$ matrices with real formal Laurent series coefficients.  If $g =(g_{ij})_{i,j=1}^n \in G$ where $g_{ij}= \sum_k g_{ij}^k t^k$, we let $Y = Y(g)= (y_{rs})_{r,s\in \Z}$ denote the infinite periodic matrix defined by $y_{i+kn,j+k'n} = g_{ij}^{k'-k}$ for $1 \leq i,j\leq n$.  For the purposes of this paper we shall nearly always think of formal loop group elements as infinite periodic matrices.  The unipotent loop group $U \subset G$ is defined as those $Y$ which are upper triangular, with 1's along the main diagonal.   

For $k \in \Z/n\Z$ and $a \in \R$, define the Chevalley generator $u_k(a) \in U$ by
$$
(u_k(a))_{ij} = \begin{cases} 
1& \mbox{if $j=i$}\\
a & \mbox{if $j=i+1=k+1$} \\
0& \mbox{otherwise.}
\end{cases}
$$
These are the standard one-parameter subgroups of $U$ corresponding to the simple roots.  They satisfy the relation
\begin{equation}\label{E:chevrelation}
u_k(a)\, u_{k+1}(b)\, u_k(c) = u_{k+1}(bc/(a+c))\, u_k(a+c)\, u_{k+1}(ab/(a+c)).
\end{equation}
For a reduced word $\i = i_1 i_2 \cdots i_\ell$ of $w \in \tS_n$ and a sequence $\a = (a_1,\ldots,a_\ell)$ of parameters, we define
$u_\i(\a) = u_{i_1}(a_1)\cdots u_{i_\ell}(a_\ell)$.  

\begin{prop}\label{prop:Bruhatcell} \cite[Theorem 3.4, Lemma 3.1]{LP3}
The map $u_\i$ is injective when restricted to positive parameters $\a \in \R_{>0}^{\ell}$ .  
Furthermore, for two reduced words $\i$ and $\j$ of the same $w \in \tS_n$, we have $U_{\geq 0}^w:=u_\i(\R_{>0}^\ell) = u_\j(\R_{>0}^\ell)$.
\end{prop}  

We call $U_{\geq 0}^w$ a TNN {\it Bruhat cell}.  It follows from the results of \cite{LP3} that for each $\i$, the map $u_\i: \R_{>0}^\ell \to U_{\geq 0}^w$ is a bijection.

\subsection{Whirl and curl matrices}
For $(x^{(1)},x^{(2)},\ldots,x^{(n)}) \in \R^n$, define the {\it {whirl} (matrix)} $M(x) = M(x^{(1)},x^{(2)},\ldots,x^{(n)})$ to be the infinite periodic matrix
$$
M(x)_{i,j} = 
\begin{cases}
1 & \text{if $j=i$;}\\
x^{(i)} & \text{if $j=i+1$;}\\
0 & \text{otherwise.}
\end{cases}
$$
and the {\it {curl} (matrix)} $N(x) =N(x^{(1)},x^{(2)},\ldots,x^{(n)})$  \cite{LP} to be the infinite periodic matrix with 
$$
N(x)_{i,j} = 
\begin{cases}
1 & \text{if $j=i$;}\\
x^{(i)} x^{(i+1)} \cdots x^{(j-1)} & \text{if $j>i$;}\\
0 & \text{otherwise.}
\end{cases}
$$
Here and elsewhere the upper indices are to be taken modulo $n$.  Whirl and curl matrices commute as follows \cite[Section 6]{LP}.  Recall the definition of $\k_{r}(x,y)$ from \eqref{eq:kappa}.
Define an algebra morphism $s: \C(x,y) \to \C(x,y)$ by
\begin{align*}
s(x^{(r)}) = \frac{y^{(r+1)}  \k_{r+1}(x,y)}{\k_{r}(x,y)} \text{\;\;\;\; and \;\;\;\;} s(y^{(r)}) = \frac{x^{(r-1)} \k_{r-1}(x,y)}{\k_{r}(x,y)}.
\end{align*}

Then we have $M(x)M(y) = M(s(x))M(s(y))$, and $N(y) N(x) = N(s(y)) N(s(x))$.  Furthermore, we have $M(x) N(y) = N(x')M(y')$, where
$$
y'_i = \frac{x_{i+1}(x_i+y_i)}{x_{i+1}+y_{i+1}} \qquad x'_i=\frac{y_{i+1}(x_i+y_i)}{x_{i+1}+y_{i+1}}.
$$

These relations are related to the transformations of whirl wires and curl wires as follows: for an $n$-tuple $x=(x^{(1)},\ldots,x^{(n)})$, let $\shift{x} = (x^{(n)},x^{(1)},\ldots,x^{(n-1)})$ denote its rotation. Then comparing with \eqref{eq:whurl} and \eqref{eq:whirlcurl}-\eqref{eq:whirlcurl2}, one obtains:

\begin{proposition}\label{prop:shift} In the following we assume we have two wire cycles $C$, $C'$ and that horizontal wires go from left to right (meeting $C$ first, then $C'$).
\begin{enumerate}
\item
Suppose two whirl wire cycles $C$, $C'$ have weights $x$ and $y$.  Then the whurl relation for $C$ and $C'$ (giving weights $x'$ and $y'$) corresponds to the commutation relation $M(x)M(\shift{y}) = M(x')M(\shift{y'})$.
\item
Suppose two curl wire cycles $C$, $C'$ have weights $x$ and $y$.  Then the whurl relation for $C$ and $C'$ (giving weights $x'$ and $y'$) corresponds to the commutation relation $N(\shift{x})N(y) = N(\shift{x'})N(y')$.
\item
Suppose we have a whirl wire cycle $C$ and a curl wire cycle $C'$ have weights $x$ and $y$.  Then the whirl-curl relation for $C$ and $C'$ (giving weights $x'$ and $y'$) corresponds to $M(\shift{x})N(\shift{y})=N(x')M(y')$. 
\item
Suppose we have a curl wire cycle $C$ and a whirl wire cycle $C'$ have weights $x$ and $y$.  Then the whirl-curl relation for $C$ and $C'$ (giving weights $x'$ and $y'$) corresponds to $N(x)M(y)=M(\shift{x'})N(\shift{y'})$. 
\end{enumerate}
\end{proposition}

\subsection{Total positivity in the rational loop group}
An element  $Y \in U$ is totally nonnegative (TNN) if every minor of $Y$ is nonnegative.  We denote by $U_{\geq 0}$ the semigroup of
totally nonnegative elements (necessarily real).  Let $U^{\rat} \subset U$ denote the subgroup of $U$ consisting of rational loops, that is, those $n \times n$
matrices $g = (g_{ij})$ with coefficients which are rational functions $P(t)/Q(t)$, and which lie in $U$.  Let $U^{\rat}_{\geq 0} = U_{\geq 0} \cap U^{\rat}$.

\begin{theorem} \label{thm:genr}
The semigroup $U^{\rat}_{\geq 0}$ is generated by whirls matrices, curl matrices, and Chevalley generators with nonnegative real parameters.
\end{theorem}
\begin{proof}
It is clear that $U^{\rat}_{\geq 0}$ is a semigroup.  Since each of the matrices $M(x^{(1)},\ldots,x^{(n)})$, $N(x^{(1)},\ldots,x^{(n)})$, and $e_i(x)$ lie in $U^{\rat}_{\geq 0}$, one inclusion is clear.  For the other inclusion, suppose that
$g \in U^{\rat}_{\geq 0}$, with corresponding infinite periodic matrix $Y$.  Put all the matrix coefficients of $g(t)$ over a common denominator $Q(t)$.  Then for each $i,j$ satisfying $j-i \gg 0$ we have that $y_{i,j}, y_{i,j+n}, y_{i,j+2n},\ldots,y_{i,j+rn}$ satisfy
a linear relation given by the coefficients of $Q(t)$, where $r$ is equal to the degree of $Q(t)$.  For a sufficiently large $j$, this would be true for $i \in \{i_0, i_0+1,\ldots, i_0 + r\}$, giving a $(r+1) \times (r+1)$ minor of $Y$ which vanishes, but which does not identically vanish on $U$.  Thus, in the terminology of \cite{LP}, $Y$ is a totally nonnegative element which is not totally positive.  Thus by \cite[Theorem 5.7]{LP}, $Y$ is a product of whirls, curls, and Chevalley generators.
\end{proof}

\begin{remark}
This notion of rationality is compatible with the one of Conjecture \ref{conj:rat}.
\end{remark}

\subsection{Whirl-curl-Bruhat cells}
Fix, $w \in \tS_n$, and two nonnegative integers $a, b \in \Z_{\geq 0}$.  Define the {\it whirl-curl-Bruhat cell} $U_{\geq 0}^{a,b,w}$ to be the image of $(\R_{>0}^n)^a \times (\R_{>0}^n)^b \times \R_{>0}^\ell$
under the map
\begin{equation}\label{E:wcbcell}
\Phi^{a,b,\i}: ((\a_1,\ldots,\a_a),(\b_1,\ldots,\b_b),(c_1,\ldots,c_\ell)) \longmapsto \end{equation} \begin{equation*}
 N(\a_1)N(\a_2)\cdots N(\a_a)  u_{i_1}(c_1)\cdots u_{i_\ell}(c_\ell)  M(\b_1)M(\b_2)\cdots M(\b_b) 
\end{equation*}
where $\i$ is a reduced word for $w$.  By Proposition \ref{prop:Bruhatcell}, $U_{\geq 0}^{a,b,w}$ depends only on $a,b,w$ and not on $\i$.  However, $\Phi^{a,b,\i}$ is not injective.  Let us define $\Omega^{a} \subset (\R_{>0}^n)^a$ to be the tuples $(\a_1,\ldots,\a_a)$ satisfying $\prod_{i \in \Z/n\Z} \a_1^{(i)} \geq \prod_{i \in \Z/n\Z} \a_2^{(i)} \geq \cdots \geq \prod_{i \in \Z/n\Z} \a_a^{(i)}$.  Denoting the restriction of $\Phi^{a,b,\i}$ to $\Omega^a \times \Omega^b \times \R_{>0}^\ell$ by $\Phi^{a,b,\i}$ as well, we have

\begin{theorem}\label{thm:Phiwhirlcurl}
The map 
$$\Phi^{a,b,\i}:\Omega^a \times \Omega^b \times \R_{>0}^\ell \longrightarrow U^{a,b,w}_{\geq 0}$$ is a bijection for every reduced word $\i$ of $w$.
\end{theorem}
\begin{proof}
According to \cite[Theorem 8.3]{LP} each of the three factors $N(\a_1)N(\a_2)\cdots N(\a_a)$, $u_{i_1}(c_1)\cdots u_{i_\ell}(c_\ell)$, and $M(\b_1)M(\b_2)\cdots M(\b_b)$ can be recovered uniquely, where the parameters of the factors $N(\a_1)N(\a_2)\cdots N(\a_a)$ and $M(\b_1)M(\b_2)\cdots M(\b_b)$ lie in $(\R_{>0}^n)^a \times (\R_{>0}^n)^b$. Then by \cite[Proposition 8.2]{LP} among those there is a unique presentation with parameters belonging to $\Omega^a \times \Omega^b$. As for the middle factor $u_{i_1}(c_1)\cdots u_{i_\ell}(c_\ell)$, it is injective according to Proposition \ref{prop:Bruhatcell}.
\end{proof}

\begin{remark}
Each $\Omega^a$ is parametrized by a semi-algebraic open-closed cell $\R_{\geq 0}^{a-1} \times \R_{>0}^{a(n-1)+1}$ given by the coordinates $r_1-r_2,r_2-r_2,\ldots,r_{a-1}-r_a$ and $r_a,\a_1^{(1)},\ldots,\a_1^{(n-1)},\ldots,\a_a^{(1)},\ldots,\a_a^{(n-1)}$, where $r_j = \prod_{i \in \Z/n\Z} \a_j^{(i)}$.  Thus it makes sense to call $U_{\geq 0}^{a,b,w}$ a ``cell'', though it is not an open cell in the usual context of cell decompositions.
\end{remark}

\begin{theorem}\label{thm:whirlcurlBruhatdecomp}
We have a disjoint union $U^{\rat}_{\geq 0} = \sqcup_{a,b,w} U^{a,b,w}_{\geq 0}$ where $a,b \in \Z_{\geq 0}$ and $w \in \tS_n$.
\end{theorem}
\begin{proof}
By Theorem \ref{thm:genr} we know that the cells $U^{a,b,w}_{\geq 0}$ cover the whole $U^{\rat}_{\geq 0}$. On the other hand, according to \cite[Theorem 8.3]{LP} those cells must be disjoint, since one can determine $a$, $b$ and $w$ uniquely from an element of $U^{a,b,w}_{\geq 0}$.
\end{proof}

\subsection{Boundary measurements as a map to the unipotent loop group} \label{ssec:tounipotent}
Let $\N$ be a simple-crossing oriented network on the cylinder with $n$ horizontal wires, all of which go from left to right.  In other words, $\N$ has $n$ boundary sources which we can cyclically label $1,2,\ldots,n$. Consider the snake path $p$ in $\N$ that starts at vertex $1$, and let $1'$ be the endpoint of $p$.  Let $1',2',\ldots,n'$ cyclically label the boundary sinks of $\N$.  Fix a continuous path $\h$ on the cylinder going from between $n$ and $1$ to between $n'$ and $1'$ staying closely above the snake path $p$ from $1$ to $1'$. Suppose a path $q$ in $\N$ crosses $\h$ from top to bottom $d$ times, and from bottom to top $d'$ times.  Then we say that the winding number $w(q)$ is equal to $d - d'$.

\begin{figure}[h!]
    \begin{center}
    \input{wire16.pstex_t}
    \end{center}
    \caption{A network on a cylinder with the snake path from $1$ to $1'$ shown.}
    \label{fig:wire16}
\end{figure}

We now define a boundary measurement matrix $M(\N) = (M(\N)_{i,j})_{i,j \in \Z}$ in the unipotent loop group from $\N$, by
$$
M(\N)_{i,j} = M^{[q(i,j',r)]}_{i,j'}(\N) \qquad \mbox{if $j-i=s+rn$, where $r \in \Z$ and $s\in \{1,2,\ldots,n\}$}
$$
where $q(i,j',r)$ is a path from $i$ to $j'$ with winding number $r$.  

\begin{lemma}
The boundary measurement $M(\N)$ lies in $U$.
\end{lemma}
\begin{proof}
We need to show that $M(\N)_{i,j} = 0$ if $i > j$ and $M(\N)_{i,i} = 1$.

Draw a snake path $q_i$ starting at every source $i=1, \ldots, n$.  Then in any highway path if one records through which of the $q_i$-s one goes, the $i$-s can only decrease by one or stay unchanged. In other words, the graph $H_N$ of Section \ref{ssec:torus} for this network is a directed path on $n$ vertices. This implies that the only way a highway path can cross $\h$ is from top to bottom. It also implies that the only path from $i$ to $i'$ that has zero winding number is $q_i$ itself. It remains to observe that each of the $q_i$-s has unit weight.
\end{proof}

\begin{example}
The network in Figure \ref{fig:wire16} produces the following element of the unipotent loop group.  Here $p,q,\ldots$ are all vertex weights.
 $$
M(N) =  \left(\begin{array}{ccccccc} \ddots & \vdots & \vdots & \vdots & \vdots &\vdots \\
\cdots& 1& p+t+w & (p+t+w)u & (p+t+w)uv & (p+t+w)uvw& \cdots \\ 
\cdots& 0& 1& q+u & qs+qv+uv & (qs+qv+uv)w & \cdots \\
 \cdots& 0&0 & 1& r+s+v& rt+rw+sw+vw&\cdots \\
 \cdots& 0&0&0&1& p+t+w& \cdots \\
  \cdots& 0&0&0&0&1&  \cdots \\
 & \vdots & \vdots & \vdots & \vdots &\vdots &\ddots
\end{array} \right)
$$
\end{example}

\begin{lemma}\label{lem:compose}
Let $\N$ and $\N'$ be two networks on the cylinder as above.  Let $\N \circ \N'$ be the cylindric network obtained by gluing two cylinders together along a boundary component, identifying the $i$-th sink of $\N$ with the $i$-th source of $\N'$.  Then $M(\N \circ \N') = M(\N)M(\N')$.
\end{lemma}

\begin{prop}
Suppose all vertex weights of $\N$ are nonnegative real numbers.  Then $M(\N) \in U_{\geq 0}$.  
\end{prop}
\begin{proof}
Follows immediately from Theorem \ref{thm:tpd}: the determinants $\det(A_{I,J})$ are exactly the minors of $M(\N)$.
\end{proof}

Theorem \ref{thm:bm} implies 
\begin{corollary}\label{cor:MN}
The map $\N \mapsto M(\N)$ to the unipotent loop group is compatible with local transformations.  That is, $M(\N) = M(\N')$ if $\N$ and $\N'$ are related by local transformations.
\end{corollary}

We shall now establish, in Section \ref{ssec:whirlcurlnetworks}, the following result.
\begin{theorem}\label{thm:ratloop}
Suppose all vertex weights of $\N$ are nonnegative real numbers.  Then $M(\N) \in U^\rat_{\geq 0}$.    Conversely, every $Y \in U^\rat_{\geq 0}$ is represented $Y = M(\N)$ by some oriented network $\N$ with nonnegative real vertex weights.
\end{theorem}

\subsection{Whirl-curl Bruhat networks}\label{ssec:whirlcurlnetworks}
Let $\N$ be an oriented simple-crossing network as in Section \ref{ssec:tounipotent}.  We shall further assume that as a wiring diagram $\N$ is reduced, as in Section \ref{ssec:wires}.  This implies that every wire cycle in $\N$ loops around the cylinder once (in one of the two directions).  Let us suppose now that $\N$ has $b$ whirl wire cycles, and $a$ curl wire cycles.  

Given a wire cycle $C$ with weights $(x^{(1)}, \cdots, x^{(n)})$ we call $\prod_i x^{(i)}$ the {\it radius} of $C$.  The following observation is immediate.

\begin{lemma}\label{lem:whurlradius}\
\begin{enumerate}
\item
When the whurl transformation is applied to two whirl cycles (resp. curl cycles) $C$ and $C'$, the radii of the two whirl cycles (resp. curl cycles) are preserved, but swapped.
\item
When a Yang-Baxter move is applied to a whirl (resp. curl) and another crossing, the radius of the whirl (resp. curl) is preserved.
\end{enumerate}
\end{lemma}

\begin{theorem} \label{thm:ntol}
Suppose that all vertex weights of $\N$ are positive real numbers.  Then $M(\N) \in U^{a,b,w}_{\geq 0}$ where $w \in \tS_n$ is given by $w(i) = j+rn+a-b$ if the wire starting at $i \in \{1,2,\ldots,n\}$ ends at $j'\in \{1',2',\ldots,n'\}$, and has winding number $r$.  Furthermore, $M(\N)$ completely determines the vertex weights of $N$ up to local transformations.
\end{theorem}
In the situation of Theorem \ref{thm:ntol}, we say that the underlying network $N(\N)$ is a $N^{a,b,w}$-network, or occassionally use $N^{a,b,w}$ to denote such a network.  We call these {\it whirl-curl-Bruhat} networks.

\begin{proof}
Let us rotate the cylinder so that sources are on the left, and sinks are on the right.  Using local transformations, in particular the Yang-Baxter move, and the whirl-curl transformation, let us move all curl cycles to the left, all whirl cycles to the right, and leave all the intersections of horizontal wires in the middle.  In addition, let us use the whurl transformations (see Lemma \ref{lem:whurlradius}) to order the curls (resp. whirls) from left to right in order of decreasing radii.  We call this the {\it canonical form} of $\N$.  

Using Lemma \ref{lem:compose}, and taking into account the shifting of Proposition \ref{prop:shift}, we see that $M(\N)$ can be obtained by substituting the vertex weights of $\N$ into the map $\Phi^{a,b,\i}$ of Theorem \ref{thm:Phiwhirlcurl}, where $\i$ is the reduced word of $w$ obtained by linearly ordering the intersections of horizontal wires.  This implies that $M(\N) \in U^{a,b,w}_{\geq 0}$, and the bijectivity of Theorem \ref{thm:Phiwhirlcurl} implies the last statement.
\end{proof}

\begin{example}
 The boundary measurement matrix $M(\N)$ of the network in Figure \ref{fig:wire16} belongs to $U^{1,1,s_0s_2}_{\geq 0}$.
\end{example}

The following results follow from Theorems \ref{thm:Phiwhirlcurl} and \ref{thm:ntol} and Corollary \ref{cor:MN}.
\begin{corollary}\label{cor:Ncanonical}
The canonical form of $\N$ is unique.  In other words, any two networks in canonical form, both of which can be obtained from $\N$ via local transformations, have the same vertex weights.
\end{corollary}

\begin{corollary}\label{cor:param}
Let $N$ be a $N^{a,b,w}$-network with all curls to the left and all whirls to the right.  The set of weighted networks $\N$ with positive real weights such that $N = N(\N)$ that are in canonical form 
give a parametrization of the whirl-curl-Bruhat cell $U^{a,b,w}$ via the map $\N \mapsto M(\N)$.
\end{corollary}

In Corollary \ref{cor:param}, the condition that all curls are to the left and all whirls to the right is not important.  What is important is that the radii of the curls and whirls have been arranged in a specified order.

\begin{remark}
A $N^{0,0,w}$-network is essentially a wiring diagram for a reduced word of $\tS_n$.  If $w \in S_n$ is a usual permutation, then these are essentially the wiring diagrams of \cite{BFZ,FZ}.  In this case, the network can be drawn on a plane rather than a cylinder.
\end{remark}

\subsection{Monodromy group of whirl-curl-Bruhat networks}\label{ssec:monodromy}

\begin{thm}\label{thm:SaSb}
Let $N$ be a $N^{a,b,w}$-network.  The monodromy group of $N$ is exactly $S_a \times S_b$.
\end{thm}

In the following we shall work with positive real weights but this is not crucial.  See for example the proof of \cite[Theorem 4.1]{LP4}.

Suppose that $\N$ is a reduced oriented simple-crossing network as before and that $N = N(\N)$ is a $N^{a,b,w}$-network.  Suppose $\N'$ is obtained from $\N$ by Yang-Baxter moves, whirl-curl moves, or whurl moves in such a way that $N(\N) = N(\N')$.  Thus $\N'$ is obtained from $\N'$ via the monodromy action, but we disallow arbitrary local transformations for now.

We define two permutations $u \in S_a$ and $v \in S_b$ as follows.  Color the curl cycles of $\N$ with colors $1,2,\ldots,a$.  When we go from $\N$ to $\N'$ the colored wire cycles may have been permuted.  Let $u$ denote this permutation, and similarly define $v$ for whirl cycles.  In this case, we say that $\N$ and $\N'$ are related by the pair $(u,v) \in S_a \times S_b$.  

\begin{lemma}\label{lem:SaSb}
Suppose the vertex weights of $\N$ are known.  Then the vertex weights of $\N'$ depend only on $(u,v) \in S_a \times S_b$.  In particular, if $(u,v) = (\id,\id)$ then $\N$ and $\N'$ have identical vertex weights. \end{lemma}
\begin{proof}
Since the Yang-Baxter moves and whirl-curl transformations are invertible, we may assume all curl cycles are to the left, and all whirl cycles to the right in $\N$ and $\N'$, as in Theorem \ref{thm:ntol}.  By applying specific whurl transformations to $\N'$ according to $(u^{-1},v^{-1})$, we may assume that $(u,v) = (\id,\id)$.  After applying another set of whurl transformations to both $\N$ and $\N'$, we may assume that the radii of the curls (resp. whirls) in $\N$ decrease from left to right.  By Lemma \ref{lem:whurlradius}, the same holds for $\N'$.  But then both $\N$ and $\N'$ are in canonical form so by Corollary \ref{cor:Ncanonical}, $\N$ and $\N'$ must have the same vertex weights.
\end{proof}

\begin{proof}[Proof of Theorem \ref{thm:SaSb}]
Suppose $\N$ and $\N'$ are related by an arbitrary sequence of local transformations, and $N(\N) = N = N(\N')$.  Then $M(\N) = M(\N')$.  By Corollary \ref{cor:Ncanonical}, the canonical forms of $\N$ and $\N'$ are identical.  It follows that $\N$ and $\N'$ are related by the $S_a \times S_b$ monodromy action in Lemma \ref{lem:SaSb}. To see that the group $S_a \times S_b$ acts faithfully we may pick vertex weights so that the radii of wire cycles are all distinct.  Lemma \ref{lem:whurlradius} then implies that the orbit under the monodromy action is of size $a! b!$.
\end{proof}

We could have avoided Lemma \ref{lem:SaSb} if we defined the pair $(u,v) \in S_a \times S_b$ by looking at the final permutations of the radii of the wire cycles.  However, we found it elegant to keep track of the wires through a sequence of local transformations.

\begin{remark}
The argument of Lemma \ref{lem:SaSb} gives another proof of Theorem \ref{thm:braid} in the case of whurl transformations which involve only whirls, or only curls.
\end{remark}

\section{Loop symmetric functions} \label{sec:lsym}
\subsection{Loop symmetric functions}
We consider a network $N^{m,m'} = N^{m,m',\id}$ (see Section \ref{ssec:whirlcurlnetworks}) on a cylinder that consists of $m$ whirl cycles, $m'$ curl cycles and $n$ horizontal wires. 
There are several different $N^{m,m',\id}$ networks, one for each relative order of whirls and curls.  Let the sources be cyclically indexed $1$ through $n$ in the order opposite to 
the direction of the whirl cycles, and let $q_r$, $r=1,\ldots,n$ be the snake path that starts at the source $r$. Let $\N^{m,m'}$ be the weighted network with $x_i^{(r)}$ the vertex 
weight assigned to the unique vertex that lies on the $r$-th snake path, the $(r+1)$-st snake path and the $i$-th whurl. Here the snake path indices are taken modulo $n$. 
Then $X = \left(x_i^{(r)}\right)_{1 \leq i \leq m, \; r \in \Z/n\Z}$ form a rectangular (or rather, cylindrical) array of variables. We caution that we will be using a different 
indexing of vertex weights in Section \ref{sec:crystals}.

Define the {\it {loop elementary symmetric functions}} in these variables by:
\begin{align*}
e_k^{(r)} &= \sum_{\stackrel{1 \leq i_1 \leq i_2 \leq \cdots \leq i_k\leq m}{ i_j < i_{j+1} \text{if $j$-th whurl is a whirl}}} x_{i_1}^{(r)} x_{i_2}^{(r+1)} \cdots x_{i_k}^{(r+k-1)}
\end{align*}

\begin{proposition}
We have $$M(\N^{m,m'})_{r,r+k} = e_k^{(r)}.$$ 
In other words, the loop elementary symmetric functions are exactly the boundary measurements for $\N^{m,m'}$.
\end{proposition}
\begin{proof}
Follows from the proof of Theorem \ref{thm:ntol} and a statement similar to \cite[Lemma 7.3]{LP}, \cite[Lemma 7.5]{LP}.
\end{proof}

Denote by $\LSym$ the ring generated by the $e_k^{(r)}$, called the ring of {\it {loop symmetric functions}}. Note that this version of the ring is a common generalization of 
{\it {whirl loop symmetric functions}} and {\it {curl loop symmetric functions}} of \cite{LP}. The {\it {loop Schur functions}} are the minors of $M(N^{m,m'})$ and have a presentation 
in terms of semistandard tableaux like the usual Schur functions, generalizing \cite[Theorem 7.4]{LP}, \cite[Theorem 7.6]{LP}. 

Namely, let each index $1 \leq j \leq m+m'$ be either  {\it {column strict}} or {\it {row strict}}, depending on whether the $j$-th whurl in $N^{m,m'}$ is a whirl or a curl (respectively).
A square $s = (i,j)$ in the $i$-th row and $j$-th column of a Young diagram has {\it
content} $i - j$ and has {\it residue} $r(s) = \overline{i- j} \in
\Z/n\Z$. A {\it {semistandard Young tableaux}} $T$ with shape
$\lambda$ is a filling of each square $s \in \lambda$ with an
integer $T(s) \in \Z_{> 0}$ so that it is weakly increasing along rows and columns, and two cells with filling $j$ are forbidden to belong to the same column (resp. row) if $j$ is column-strict (resp. row-strict). The weight $x^T$ of a
tableaux $T$ is given by $x^T = \prod_{s \in \lambda}
x_{T(s)}^{(r(s))}$.  We define the loop Schur function by
$$
s_\lambda = \sum_{T} x^T
$$
where the summation is over all semistandard Young tableaux of
(skew) shape $\lambda$.

\begin{theorem} \label{thm:schur}
Let $I = i_1 < i_2 < \ldots < i_k$ and $J = j_1 < j_2 < \ldots < j_k$ be two sequences of
integers such that $i_t \leq j_t$.   Define
$$\lambda = \lambda(I,J) = (j_k, j_{k-1} + 1, \ldots,
j_1+k-1)/(i_k,i_{k-1}+ 1 \ldots, i_1+k-1).$$ Then
$$
\det(e_{j_t - i_s}^{(i_s)})_{s,t=1}^k =
s_{\lambda'}.
$$
\end{theorem}
\begin{proof}
A Gessel-Viennot type argument, straight forwardly generalizing proofs of \cite[Theorem 7.4]{LP}, \cite[Theorem 7.6]{LP}. 
\end{proof}

The loop symmetric functions are the invariants of the monodromy group action on $N^{m,m'}$ of Theorem \ref{thm:SaSb}.  In this case, 
the monodromy group $S_m \times S_{m'}$ acts on $\C(N^{m,m'}) = \C(x_i^{(r)})$ by birational transformations given by the whurl relations \eqref{eq:whurl}.  In the special case of 
either $m$ or $m'$ equal to $0$ this birational transformation appeared in \cite{Y}, see also \cite{LP3,LP4}.

\begin{theorem} \cite[Theorem 4.1]{LP4} \label{thm:LP4}
If $m'=0$ ($m=0$) the loop elementary symmetric functions generate the field $\mathbb C(X)^{S_m}$ ($\mathbb C(X)^{S_{m'}}$) of invariants of the monodromy $S_m$-action ($S_{m'}$-action).
\end{theorem}

\begin{remark}
We shall show in \cite{LPLSym} that a stronger statement holds.  Namely in the case $m=0$ or $m'=0$, the ring $\LSym$ is exactly the ring of {\it {polynomial}} invariants of the birational action of $S_m$ or $S_{m'}$ via whirl (curl) moves.
\end{remark}

\subsection{Power sum loop symmetric functions and cycle measurements}
We discuss only $N^{0,m}$ here.  The story for $N^{m,0}$ is analogous.  

The network $\N^{0,m}$ has no non-trivial (highway) cycle measurements. However, its underway cycle measurements can be naturally described as follows. Denote by $p_k$ the underway cycle measurement which has homology class of the $k$-fold loop around the cylinder (in the same direction as the whirl cycles).

\begin{proposition}
 We have $$p_k = \frac{1}{k} \sum_{i=1}^m (\prod_{r=1}^n x_i^{(r)})^k.$$
\end{proposition}

\begin{proof}
The only way to achieve the desired homology of an underway cycle is to circle around one of the whirls $k$ times.
\end{proof}

We call the $p_k$ the {\it {(rescaled) power sum loop symmetric functions}}. Such terminology is justified by the following proposition, which is established in \cite{LPLSym}.

\begin{proposition} \label{prop:lp}
The $p_k$ lie in $\LSym$.
\end{proposition}

It is interesting to observe that for the network $\N^{0,m}$ Proposition \ref{prop:lp} has the following meaning: the cycle measurements are (polynomial) functions of the boundary measurements. In general such a statement would not be true: on a torus there are no boundary measurements at all, while definitely there are non-trivial cycle measurements.

The classical {\it {problem of moments}} studies the question of recovering a probability distribution from its moments, and in particular determining what conditions the moments must satisfy for a solution to exist. See \cite{ShT} for a detailed account. For {\it {discrete}} distributions taking nonnegative real values, the latter question is equivalent to determining the allowed range of power sum symmetric functions of a sequence of nonnegative variables. For a finite sequence the condition that the elementary symmetric functions form a totally positive sequence \cite{Edr,Br} gives inequalities which characterize the possible range of power sum symmetric functions.

In our current language, this is the question of describing the possible cycle measurements of $\N^{0,m}$ (where the problem of moments corresponds to $n = 1$), as the vertex weights range over nonnegative real numbers.  

\begin{problem}
What is the range of cycle measurements when we allow an arbitrary network with nonnegative real weights on a cylinder?
\end{problem}
See Theorem \ref{thm:tpd} and Conjecture \ref{conj:tpd} for the case of boundary measurements.

\part{Geometric crystals on cylinders and tori}
\section{Affine geometric crystals}\label{sec:crystals}
\subsection{Geometric crystals}\label{sec:geom}
We shall use \cite{KNO} as our main reference for affine geometric crystals.
In this paper we shall consider affine geometric crystals of type $A$.  Fix $n > 1$.  Let $A= (a_{ij})_{i,j \in \Z/n\Z}$ denote the $A_{n-1}^{(1)}$ Cartan matrix.  Thus if $n > 2$ then $a_{ii} = 2$, $a_{ij} = -1$ for $|i-j|=1$, and $a_{ij}=0$ for $|i-j| > 1$.  For $n = 2$, we have $a_{11} = a_{22}=2$ and $a_{12}=a_{21} = -2$.  

Let $(X,\{e_i\}_{i \in \Z/n\Z},\{\ep_i\}_{i\in\Z/n\Z},\{\gamma_i\}_{i \in \Z/n\Z})$ be an affine geometric crystal of type $A_{n-1}^{(1)}$ (that is, for $\uqsln$).  Thus, $X$ is a complex algebraic variety, $\ep_i: X \to \C$ and $\gamma_i: X \to \C$ are rational functions, and $e_i: \C^{*} \times X \to X$ $((c,x) \mapsto e_i^c(x))$ is a rational $\C^{*}$-action, satisfying:
\begin{enumerate}
\item
The domain of $e_i^1: X \to X$ is dense in $X$ for any $i \in \Z/n\Z$.
\item
$\gamma_j(e_i^c(x)) = c^{a_{ij}} \gamma_j(x)$ for any $i,j \in \Z/n\Z$.
\item
$\ep_i(e_i^c(x)) = c^{-1}\ep_i(x)$.
\item 
for $i \neq j$ such that $a_{ij} = 0$ we have
$e_i^c e_j^{c'} = e_j^{c'} e_i^{c}$.
\item
for $i \neq j$ such that $a_{ij} = -1$ we have
$e_{i}^c e_{j}^{cc'} e_{i}^{c'} = e_j^{c'} e_i^{cc'} e_{j}^c$.
\end{enumerate}
We often abuse notation by just writing $X$ for the geometric crystal.  We define $\ph_i = \gamma_i \ep_i$, and sometimes define a geometric crystal by specifying $\ph_i$ and $\ep_i$, instead of $\gamma_i$.

\subsection{Weyl group action}\label{ssec:Weyl}
One obtains \cite{BK} a rational action of the affine symmetric group on $X$ by setting $s_i = e_i^{\gamma_i^{-1}}$.  Then $\{s_i \mid i \in \Z/n\Z\}$ satisfy the relations of the simple generators of $\tS_n$.

\subsection{Products}
If $X, X'$ are affine geometric crystals, then so is $X \times X'$ \cite{BK, KNO}.  Let $(x,x') \in X \times X'$.  Then
\begin{equation}
\label{E:epphproduct}
\ep_k(x,x')= \frac{\ep_k(x) \ep_k(x')}{\ph_k(x')+\ep_k(x)} \qquad
\ph_k(x,x')=\frac{\ph_k(x) \ph_k(x')}{\ep_k(x)+ \ph_k(x')}
\end{equation}
and
\begin{align*}
e_k^c(x \otimes x') = (e_k^{c^+}x,e_k^{c/c^+}x')
\end{align*}
where
\begin{equation}\label{eq:c+}
c^+ = \frac{c\ph_k(x')+\ep_k(x)}{\ph_k(x') +\ep_k(x)}.
\end{equation}

\begin{remark}
Note that our notations differ from those in \cite{KNO} by swapping left and right in the product, and this agrees with \cite{LP2}. Note also that the birational formulae we use should be tropicalized using $(\min, +)$ (rather than $(\max,+)$) operations to yield the formulae for combinatorial crystals.
\end{remark}

\subsection{The basic geometric crystals}
We now introduce a geometric crystal $X_M$ which we call the {\it basic geometric crystal} of type $A_{n-1}^{(1)}$.  This is a geometric analogue of a limit of perfect crystals.  It is essentially the geometric crystal ${\mathcal B}_L(A_{n-1}^{(1)})$ of \cite[Section 5.2]{KNO}.

We have $X_M = \{x=(x^{(1)},x^{(2)},\ldots,x^{(n)}) \mid x^{(i)} \in \C^{*}\}$ and
$$
\ep_i(x) = x^{(i+1)} \qquad \ph_i(x) = x^{(i)} \qquad \gamma_i(x) = x^{(i)}/x^{(i+1)}
$$ 
and
$$
e_i^c: (x^{(1)},x^{(2)},\ldots,x^{(n)}) \longmapsto (x^{(1)},\ldots,cx^{(i)},c^{-1} x^{(i+1)}\ldots,x^{(n)}) .
$$
That $X_M$ is an affine geometric crystal is shown in \cite{KNO}.

We also introduce the dual basic affine geometric crystal $X_N$.  We have $$X_N = \{x=(x^{(1)},x^{(2)},\ldots,x^{(n)}) \mid x^{(i)} \in \C^{*}\}$$ and
$$
\ep_i(x) = x^{(i)} \qquad \ph_i(x) = x^{(i+1)} \qquad \gamma_i(x) = x^{(i+1)}/x^{(i)}
$$ 
and
$$
e_i^c: (x^{(1)},x^{(2)},\ldots,x^{(n)}) \longmapsto (x^{(1)},\ldots,c^{-1}x^{(i)},c x^{(i+1)}\ldots,x^{(n)}).
$$
One can easily verify that $X_N$ satisfies the axioms of geometric crystals.

\section{Crystals as networks}\label{sec:crystalnetwork}
In this section we assume networks have weights in $\C$.
\subsection{Parallel wires}\label{ssec:parallel}
Consider a vertex-weighted, simple-crossing, oriented network $\N$.  Suppose $W$ and $W'$ are two wires with endpoints on the same boundary components.  We say that $W$ and $W'$ are {\it parallel} with $W$ to the left of $W'$ (and $W'$ to the right of $W$) if the part of the surface to the left of $W'$ and to the right of $W$ is a strip containing no interior vertices or isolated loops, and furthermore, every wire that crosses $W$ from left to right (resp. $W'$ from right to left) then crosses $W'$ from left to right (resp. $W$ from right to left) before intersecting any other wire. 
\begin{figure}[h!]
    \begin{center}
    \input{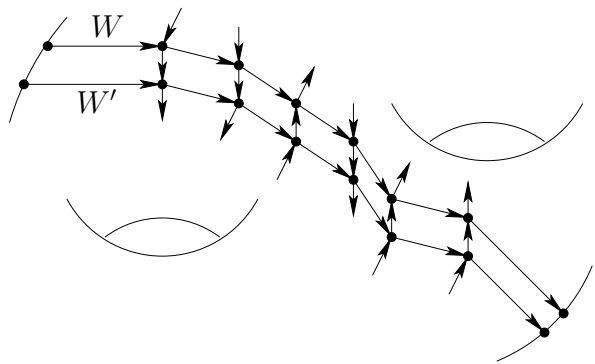}
    \end{center}
    \caption{Two parallel wires on a surface}
    \label{fig:wire30}
\end{figure}

If $W$ and $W'$ are parallel, then they contain the same number of vertices.  Let the interior vertices on $W$ have weights $x^{(1)},\ldots,x^{(m)}$ and the weights on $W'$ have weights $y^{(1)},\ldots,y^{(m)}$.  Thus the $i$-th orthogonal wire crosses $W$ and $W'$ at vertices with weights $x^{(i)}$ and $y^{(i)}$ respectively.
Now define 
$$
z^{(i)} = \begin{cases} 
x^{(i)} & \mbox{if the $i$-th orthogonal wire intersects $W'$ first}\\
y^{(i)}& \mbox{if the $i$-th orthogonal wire intersects $W$ first}
\end{cases}
$$
$$
t^{(i)} = \begin{cases} 
x^{(i)} & \mbox{if the $i$-th orthogonal wire intersects $W$ first}\\
y^{(i)}& \mbox{if the $i$-th orthogonal wire intersects $W'$ first}
\end{cases}
$$
Finally we define functions $\ep(W,W')$ (also denoted $\ep_{\N}(W,W')$ to emphasize $\N$) and $\ph(W,W')$ (also denoted $\ph_{\N}(W,W')$) by
$$\ph(W,W') = \frac{\prod_{i=1}^m z^{(i)}}{\k_1(z, \shift{t})}; \;\;\;\; \ep(W,W') = \frac{\prod_{i=1}^m t^{(i)}}{\k_1(z, \shift{t})}.$$

If the denominator is 0, then $\ep(W,W')$ and $\ph(W,W')$ are undefined.

We can define wire cycles $C, C'$ to be parallel in a similar manner.  To define $\ep(C,C')$ and $\ph(C,C')$, one has to make a choice of where to begin labeling the orthogonal wires.  We shall return to this point in Section \ref{sec:torus}.

\subsection{Wiring crystal action}\label{ssec:wiringcrystal}
Suppose $W$ and $W'$ are parallel wires in $\N$ with $W$ to the left of $W'$.  For $c \in \C^*$, we define a network $e^c(\N,W,W')$ (also denoted $e^c(\N)$ for simplicity), as follows.  Create two new crossings $v$ and $u$ between the $W$ and $W'$, with $v$ before all the interior vertices of $W$ and $W'$, and $u$ after all the interior vertices of $W$ and $W'$.  We assign vertex weights $(c-1)\ph(W,W')$ and $(c^{-1}-1)\ep(W,W')$ to $v$ and $u$ respectively.  Now we use the Yang-Baxter transformation to move the crossing at $v$ past all the orthogonal wires.  If $\ep(W,W')$ or $\ph(W,W')$ are undefined, or the Yang-Baxter moves cannot be performed (because some denominators are $0$), then $e^c(\N,W,W')$ is undefined.  We suppose this is not the case.

\begin{lem}\label{lem:cancel}
After $v$ is moved past all the orthogonal wires, its weight is equal to the negative of the weight of $u$.  Furthermore, $\ep(W,W')$ and $\ph(W,W')$ are the unique rational functions in the vertex weights of $W$ and $W'$ with this property.
\end{lem}
\begin{proof}
The first statement can be obtained by a direct calculation, but it will follow from our constructions in Sections \ref{ssec:whirlcrystal}-\ref{ssec:whirlcurlcrystal} below.  To obtain the second statement, we note that the stated property is equivalent to an equation of the form
$$
\frac{(c-1)A\ph + B}{(c-1)C\ph + D} = (1-c^{-1})\ep
$$
where $A,B,C,D$ are rational functions in the vertex weights along $W$ and $W'$.  This equation can be converted to a quadratic polynomial equation in $c$, the coefficients of which determine $\ph$ and $\ep$ in terms of $A,B,C,D$.
\end{proof}

It follows from Lemma \ref{lem:cancel} that we can apply the moves (XM) and (XR) to remove the extra vertices, leaving a network $e^c(\N,W,W')$ which is the same as $\N$ except for the vertex weights along $W$ and $W'$.  We call this the wiring crystal action on $\N$.  

\begin{remark}
In the definition of $e^c(\N,W,W')$ we could also ask for $u$ to be moved to the front to cancel with $v$, or indeed for both to be moved to any location in the middle. 
\end{remark}

We now state some basic properties of $e^c$, relating the wiring crystal action to geometric crystals.

\begin{thm} \label{thm:wiringcrystal} \
\begin{enumerate} 
\item
We have the equations
$$
\ep(e^c(W),e^c(W')) = c^{-1}\ep(W,W') \qquad \ph(e^c(W),e^c(W')) = c\ph(W,W')
$$
where abusing notation we denote by $e^c(W)$ the image of $W$ in the network $e^c(\N,W,W')$.
\item
The wiring crystal action gives a rational action of $\C^*$ on $\N$.  Thus for generic vertex weights on $W$ and $W'$, the network $e^c(\N,W,W')$ is defined, and we have $e^c(e^{c'}(\N)) = e^{cc'}(\N)$.
\item
Suppose $(W,W')$ and $(Y,Y')$ are pairs of parallel wires with no common vertices.  Then the wiring crystal actions $e^c(\cdot,W,W')$ and $e^c(\cdot,Y,Y')$ commute.
\item
Suppose we have three parallel wires $W,W',W''$ with $W$ to the left of $W'$ and $W'$ to the left of $W''$.  Then the wiring crystal actions $e = e^c(\cdot,W,W')$ and $e'= e^c(\cdot,W',W'')$ satisfy the geometric crystal braid relation:
$e^c e'^{cc'} e^{c'} = e'^{c'} e^{cc'} e'^c$.
\end{enumerate}
\end{thm}

The wiring crystal action is completely local to $W$ and $W'$.  So to establish Theorem \ref{thm:wiringcrystal} it suffices to consider two or three parallel wires on a cylinder.  The stated properties in this case will follow from our identification of certain networks on the cylinder with products of the geometric crystals $X_M$ and $X_N$.

\subsection{Weyl group action and the whurl relation}\label{ssec:Weylwhurl}
Let $(W,W')$ be parallel wire cycles in $\N$.  Then there is a whurl relation (Section \ref{ssec:whurl}) which swaps $W$ and $W'$.  On the other hand, cutting $W$ and $W'$ in the same place, we can apply a wiring crystal action $e^c(\cdot,W,W')$ to the pair $(W,W')$.  

\begin{lemma} \label{lem:crR}
Regardless of where we cut $W$ and $W'$, the Weyl group action $s_{W,W'} = e^{\gamma^{-1}}$ (Section \ref{ssec:Weyl}) coincides with the whurl relation.
\end{lemma}

\begin{proof}
 If $c=\gamma^{-1}$ then $$(c-1)\ph= \ep - \ph = \frac{\prod_{i=1}^m t^{(i)}-\prod_{i=1}^m z^{(i)}}{\k_1(z, \shift{t})}$$ which coincides with the formula for $p$ in Theorem \ref{thm:whloc}. 
\end{proof}

\subsection{The whirl and curl crystals}\label{ssec:whirlcrystal}
Consider an oriented simple-crossing network $\N$ on a cylinder that consists of a single whirl wire cycle (going around the cylinder once) crossed by $n$ horizontal wires that do not cross each other and connect $n$ source boundary vertices $1,2,\ldots,n$ on the left to $n$ sink boundary vertices $1',2',\ldots,n'$ on the right.  That is $N(\N') = N^{0,1}= N^{0,1,\id}$ in the notation of Section \ref{sec:whirlcurl}.

Assign the weight $x^{(j)}$ to the interior vertex of the network that is the intersection obtained of the whirl cycle with the $j$-th horizontal wire $W_j$.  Let $\XX_M = \{\N \mid x^{(j)} \in \C^*\}$ denote the set of such networks as the weights $x^{(j)}$ are allowed to range over $\C^*$.   The set of networks $\XX_M$ can be naturally identified with the basic affine geometric crystal $X_M$.

\begin{prop}\label{prop:wcrystal}
The functions $\ep(W_j,W_{j+1})$ and $\ph(W_j,W_{j+1})$ on $\XX_M$ agree with $\ep_j$ and $\ph_j$ of $X_M$ respectively.  The wiring crystal action $e^c(\cdot,W_j,W_{j+1})$ on $\XX_M$ agrees with the geometric crystal actions $e^c_j$ of $X_M$.
\end{prop}

\begin{figure}[h!]
    \begin{center}
    \input{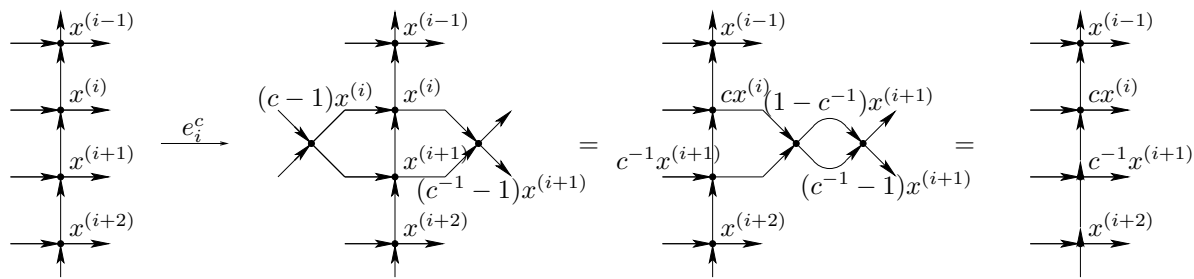}
    \end{center}
    \caption{The action of the crystal operator $e_i$ on $X_M$.}
    \label{fig:wire23}
\end{figure}

%

Instead of using a whirl wire cycle, one can use a curl wire cycle to get the network $N^{1,0} = N^{1,0,w}$.  This gives a set of networks $\XX_N$. 
\begin{prop}\label{prop:ccrystal}
The functions $\ep(W_j,W_{j+1})$ and $\ph(W_j,W_{j+1})$ on $\XX_N$ agree with $\ep_j$ and $\ph_j$ of $X_N$ respectively.  The wiring crystal action $e^c(\cdot,W_j,W_{j+1})$ on $\XX_N$ agrees with the geometric crystal actions $e^c_j$ of $X_N$.
\end{prop}

We say that $N$ is a {\it network model} for an affine geometric crystal
$X$ of type $A_{n-1}^{(1)}$, if the set $\XX(N)= \{\N\}$ (called a {\it network crystal}, or {\it wiring diagram crystal}) of weighted networks $\N$ with weights in $\C^*$ and $N = N(\N)$ can be birationally identified with $X$, such that for appropriate labellings of wires,
\begin{enumerate}
\item
The functions $\ep(W_j,W_{j+1})$ and $\ph(W_j,W_{j+1})$ on $\XX(N)$ agree with $\ep_j$ and $\ph_j$ of $X$ respectively.  
\item
The wiring crystal action $e^c(\cdot,W_j,W_{j+1})$ on $\XX(N)$ agrees with the geometric crystal actions $e^c_j$ of $X$.
\end{enumerate}
Thus Propositions \ref{prop:wcrystal} and \ref{prop:ccrystal} state that $N^{0,1}$ and $N^{1,0}$ are network models for $X_M$ and $X_N$ respectively.

\subsection{Products of network crystals}\label{ssec:whirlcurlcrystal}
Suppose $N$ and $N'$ are network models on the cylinder for affine geometric crystals $X$ and $X'$, both of type  $A_{n-1}^{(1)}$.  We define the network composition $N \circ N'$ to be the network on the cylinder obtained by gluing the sinks of $N$ to the sources of $N'$, and then removing those vertices.

\begin{prop}\label{prop:compose}
The composition $N \circ N'$ is a network model for $X \times X'$.
\end{prop}
\begin{proof}
It is clear that $N \circ N'$ is birationally isomorphic with $X \times X'$ in the obvious manner.  Pick and fix a pair of parallel wires $W_j$ and $W_{j+1}$, which will henceforth be omitted from the notation.  We first check that the $\ph$ and $\ep$ functions of $\N \circ \N'$ and $X \times X'$ agree.  We use $z,z',t,t',m,m'$ to denote the corresponding quantities for $X$ and $X'$ as in Section \ref{ssec:parallel}.  By \eqref{E:epphproduct},
$$
\ph_{X \times X'} =\frac{ \frac{\prod_{i=1}^m z^{(i)}}{\k_1(z,\shift{t})}\frac{\prod_{i=1}^{m'} z'^{(i)}}{\k_1(z',\shift{t'})}}{\frac{\prod_{i=1}^m t^{(i)}}{\k_1(z,\shift{t})}+\frac{\prod_{i=1}^{m'} z'^{(i)}}{\k_1(z',\shift{t'})}}
= \frac{\prod_{i=1}^m z^{(i)}\prod_{i=1}^{m'}z'^{(i)}}{\prod_{i=1}^m t^{(i)}\k_1(z',\shift{t'}) + \k_1(z,\shift{t})\prod_{i=1}^{m'} z'^{(i)} } = \ph_{\N\circ \N'}
$$

Let us compare the wiring crystal action $e^c(\N \circ \N')$ with the composition $e^{c^+}(\N) \circ e^{c/c^+}(\N')$.  The former is computed by introducing two new crossings with weights $(c-1)\ph_{\N\circ \N'}$ and $(c^{-1}-1)\ep_{\N \circ \N'}$.  The latter is computed by introducing four new crossings (one in the beginning, two in the middle, and one at the end) with weights $(c^+-1)\ph_\N$, $((c^+)^{-1}-1)\ep_\N)$, $(c/c^+-1)\ph_{\N'}$, and $((c/c^+)^{-1}-1)\ep_{\N'})$.  A quick calculation shows that $((c^+)^{-1}-1)\ep_\N)=-(c/c^+-1)\ph_{\N'}$ so that the middle two crossings can be canceled using (XM) and (XR).  But the remaining first and last crossings must still cancel each other out when moved next to each other.  Since $(c^+-1)\ph_\N= (c-1)\ph_{\N \circ \N'}$, it follows that $e^{c^+}(\N) \circ e^{c/c^+}(\N')$ computes the wiring crystal action $e^c(\N \circ \N')$.
\end{proof}

\subsection{Products of basic and dual basic affine geometric crystals}
Consider the product affine geometric crystal $X^{\tau} = X_{\tau_1} \times X_{\tau_2} \times \dotsc \times X_{\tau_m}$, where each $X_{\tau_k}$ is either $X_M$ or $X_N$.
It follows from Proposition \ref{prop:compose} that a network model for $X^{\tau}$ is obtained as follows.
Take the network $N^\tau$ on a cylinder that consists of whirl and curl wire cycles crossed by $n$ horizontal wires that do not cross each other and connect $n$ source boundary vertices on one side of the cylinder to $n$ sink boundary vertices on the other side.  Each wire cycle is a whirl or a curl depending on whether corresponding $X_{\tau_k}$ factor is $X_M$ or $X_N$, respectively.  Thus $N^\tau$ is a $N^{a,b,\id}$-network, where $X_{\tau_k}$ has $a$ factors  equal to $X_N$ and $b$ factors equal to $X_M$.

\begin{thm}\label{thm:whirlcurlcrystal}
The network $N^\tau$ is a network model for the geometric crystal $X^\tau$.
\end{thm}

The general description of $\ep$ and $\ph$ we give here generalizes \cite[Theorem 4.2]{LP4} for $(X_M)^m$.


%

\begin{example}
 Consider the fragment of a crystal shown in Figure \ref{fig:wire24}.  
\begin{figure}[h!]
    \begin{center}
    \input{wire24.pstex_t}
    \end{center}
    \caption{}
    \label{fig:wire24}
\end{figure}
In this case $X^{\tau} = X_M \times X_N \times X_M$ and the functions $\ph_i$ and $\ep_i$ are as follows: $$\ph_i = \frac{x_1^{(i)}x_2^{(i+1)}x_3^{(i)}}{x_2^{(i+1)}x_3^{(i)}+x_1^{(i+1)}x_3^{(i)}+x_1^{(i+1)}x_2^{(i)}}; \;\;\;\;\; \ep_i = \frac{x_1^{(i+1)}x_2^{(i)}x_3^{(i+1)}}{x_2^{(i+1)}x_3^{(i)}+x_1^{(i+1)}x_3^{(i)}+x_1^{(i+1)}x_2^{(i)}}.$$
\end{example}

\subsection{Whurl, whirl-curl relations and the $R$ matrix}
Recall from \cite[Section 9]{KNO} that the {\it $R$-matrix} of a product $X \times Y$ of (affine) geometric crystals is a birational map $R: X \times Y \to Y \times X$ which commutes with all the crystal structures, and which in addition satisfies the braid relation $$(R_{12} \times 1) \circ (1 \times R_{23}) \circ (R_{12} \times 1) = (1 \times R_{23}) \circ (R_{12} \times 1) \circ (1 \times R_{23})$$
when applied to $X \times Y \times Z$.

Let $\XX^\tau=\XX(N^\tau)$ be the network crystal for $X^\tau$.  

\begin{prop}\label{prop:Rmatrix}
The whurl move, or whirl-curl move, applied to the wire cycles in $N \in \XX^\tau$ is exactly the $R$-matrix for products of $X_M$ and $X_N$.
\end{prop}
\begin{proof}
The braid relation is already known to hold (Theorem \ref{thm:braid}).  So we need to check the claim only for network versions of $X_M^2$, $X_M \times X_N$ and $X_N^2$.  By inspection, the functions $\ep$ and $\ph$ are invariant when a whurl move or whirl-curl move is applied.  We shall see in Theorems \ref{thm:parallelcommute} and \ref{thm:antiparallelcommute} that the wiring crystal action $e^c$ commutes with the whurl move and the whirl-curl move.  
\end{proof}

\subsection{$\ep$,$\ph$ and boundary measurements}

Consider the networks $N^{0,m}$ and $N^{m,0}$ on a cylinder as in Section \ref{sec:lsym}.  They are network models for the product affine geometric crystals $X_M^m$ and $X_N^m$.  We caution that there is a shift between the labeling of the variables $x_i^{(r)}$ in Section \ref{sec:lsym} and in the current section.  The following result is established in \cite{LP4} in the case of $N^{0,m}$, and the case of $N^{m,0}$ is analogous.

\begin{prop} \label{prop:cylep} \cite[Theorem 4.2]{LP4}
 The formulae for $\ep_i$ and $\ph_i$ of $\XX(N^{0,m})$ are given by $$\phi_i = \frac{e_m^{(i)}}{e_{m-1}^{(i+1)}}; \;\;\;\; \ep_i = \frac{e_m^{(i+1)}}{e_{m-1}^{(i+1)}}$$
where $e_k^{(s)}$ denotes the loop elementary symmetric functions for $N^{0,m}$.
  The formulae for $\ep_i$ and $\ph_i$ in $\XX(N^{m,0})$ are given by $$\phi_i = \frac{ e_m^{(i+1)}}{ e_{m-1}^{(i+1)}}; \;\;\;\; \ep_i = \frac{ e_m^{(i)}}{e_{m-1}^{(i+1)}}$$
where $e_k^{(s)}$ denotes the loop elementary symmetric functions for $N^{m,0}$.
\end{prop}

Now let $N^{a,b,\id}$ be a cylindric network that has $n$ horizontal wires, $b$ whirls and $a$ curls in some order.  For $\XX(N^{a,b,\id})$, the functions $\ep$ and $\ph$ are 
no longer simple ratios of boundary measurements. However, we do have the following proposition.

\begin{prop}\label{P:epphrat}
 The $\ep$ and $\ph$ functions for $\XX(N^{a,b,\id})$ are rational functions in the boundary measurements of $N^{a,b,\id}$.
\end{prop}

\begin{proof}
 From Proposition \ref{prop:Rmatrix} we know that applying whirl-curl relations does not change $\ep$ and $\ph$, while from Theorem \ref{thm:bm} we know it does not change the measurements. 
Thus we may assume all $b$ whirls are on the left  of the cylinder, and all $a$ curls on the right. Split our network (and the cylinder) into two subnetworks
(and two subcylinders) so that all whirls belong to one, all curls to the other. If we can determine the $\ep$ and $\ph$ functions of each of the two subnetworks from the boundary 
measurements of the whole network, we can find $\ep$ and $\ph$ for the whole network, as in the formula for the product of geometric crystals.  By Proposition \ref{prop:cylep} it suffices to be 
able to determine the boundary measurements for each of the two subnetworks from the boundary measurements of the whole network.

Now we use Theorem \ref{thm:schur}. According to the chosen order of whirls and curls, the semistandard fillings have indices $1$ through $b$ 
column-strict and indices $b+1$ through $b+a$ row-strict. Let $\lambda = (a+1, \ldots, a+1)$ be the rectangular shape $b \times (a+1)$. Let $\mu = (a+1, \ldots, a+1, 1, \ldots, 1)$
be $\lambda$ with $k$ extra cells in the first column. We claim that there is a bijection between semistandard fillings of $\mu$ and pairs of semistandard fillings of $\lambda$ and of a column 
shape of length $k$, where the latter is in the alphabet $b+1, \ldots, b+a$. Indeed, in both $\lambda$ and $\mu$ the leftmost cell of the $b$-th row is forced to be filled with $b$. It cannot be filled with 
anything smaller since $1, \ldots, b$ are column strict, and there are $b-1$ cells above it.  It also cannot be filled with anything larger since $b+1, \ldots, b+a$ are row strict, and there 
are $a$ cells to the right from it. This forces the $(b+1)$-st cell in the first column of $\mu$ to be filled with at least $b+1$, but imposes no other restrictions. The bijection is then obtained
just by cutting off the last $k$ cells of the first column.

This implies that each $e_k^{(s)}$ (and thus each boundary measurement) of the curl subnetwork can be obtained as a ratio ${s_{\mu}}/{s_{\lambda}}$, and is thus a rational function of the boundary measurements of $N^{a,b,\id}$. 
The boundary measurements of the whirl subnetwork are then recovered according to Proposition \ref{prop:compose}, and are also rational functions of the boundary measurements of $N^{a,b,\id}$. The proposition follows.
\end{proof}

The energy function of an affine crystal is an important invariant which connects crytals to the theory of lattice models.  Affine geometric crystals $X$ have rational energy functions $D_X$.  We do not discuss them in detail here.

\begin{proposition}\label{P:energy}
The energy function of $\XX(N^{a,b,\id})$ are rational functions in the boundary measurements of $N^{a,b,\id}$.
\end{proposition}
\begin{proof}
The energy function is an invariant of the birational $R$-matrix.  It follows from Theorem \ref{thm:LP4} and the proof of Proposition \ref{P:epphrat} that it is a rational function in the boundary measurements of $N^{a,b,\id}$.
\end{proof}

\begin{remark}
In \cite{LP3}, an explicit formula for the rational energy function $D_X$ of $X=X_M^m$ was given: the energy is the loop Schur function $s_{\lambda}$ for a stretched staircase shape $\lambda$. 
In the more general case of $\XX^\tau$, it appears that the same is true.  Namely, there is a shape $\lambda^\tau$ such that the energy function is the loop Schur function $s_{\lambda^\tau}$, as defined 
in Section \ref{sec:lsym}.  The shape $\lambda^\tau$ depends on the number of whirl and curl factors in $\XX^\tau$, but not on their order.
\end{remark}

\subsection{Combinatorial analogue of whirl-curl affine crystals}\label{ssec:comb}
We now briefly explain the connection with combinatorial affine crystals, assuming the reader is familiar with the latter.  We shall use \cite{Sh} as our main reference for affine crystals.  A comprehensive study of the relation between geometric crystals and combinatorial crystals is given in \cite{BK2}.

We consider two kinds of affine crystals for $\uqsln$: the crystals $B_M$ which are crystals for the symmetric powers of the standard representation of $\uqsln$, and their duals, denoted $B_N$.  The affine crystals $B_M$ can be identified with the set of semistandard tableaux filled with $1,2,\ldots,n$ with a fixed row-shape.  (Note that there is one geometric crystal $X_M$, but there are many possible crystals $B_M$ depending on the length of the row.)  Denote by $x^{(j)}$ the number of $j$-s in $b \in B_M$.  Similarly, each $b \in B_N$ can be identified with a semistandard tableau with $n-1$ rows. Denote by $x^{(j)}$ the number of the columns in $b$ that do not contain the number $j$. Then the $\ep$ and $\ph$ functions of $B_M$ and $B_N$ agree with those of $X_M$ and $X_N$.  

The geometric crystal $X^{\tau} = X_{\tau_1} \times X_{\tau_2} \times \dotsc \times X_{\tau_m}$ is the geometric analogue of the tensor product combinatorial affine crystals $B^\tau = B_1 \otimes \cdots \otimes B_m$, where $B_i$ is a $B_M$ (resp. $B_N$) as to $X_{\tau_i}$ is $X_M$ (resp. $X_N$).

\begin{proposition} \label{prop:trop}
The action of the crystal operators $\tilde e_i$ on the combinatorial crystal $B^\tau$, and the operators $e_i^c$ on a geometric crystal $X^\tau$ are related by {\it tropicalization}: write $e_i^c$ as a rational function (in the $x$-s) then make the changes: $\times \to +$, $\backslash \to -$, and $+ \to \min$.  Finally change every occurrence of $c$ to a $1$.  This gives a piecewise linear formula $\trop(e_i^c)$ which agrees with $\tilde e_i$, except when the latter sends a crystal element to 0.  The combinatorial crystal operator $\tilde e_i$ sends a crystal element to 0 exactly when $\ep = 0$ on that element.
\end{proposition}
\begin{proof}
It is clear that the claim is true for $B_M$ and $X_M$ (or $B_N$ and $X_N$). For example, the tropical version of $$e_i^c: (x^{(1)},x^{(2)},\ldots,x^{(n)}) \longmapsto (x^{(1)},\ldots,cx^{(i)},c^{-1} x^{(i+1)}\ldots,x^{(n)}) $$ is that the crystal operator $\tilde e_i$ increases the number of $i$-s by one and decreases the number of $(i+1)$-s by one. An element of a single row crystal is sent to zero exactly when a crystal operator attempts to decrease the number of $i$-s or $(i+1)$-s below zero.

Now we claim that if the proposition is true for two combinatorial crystals $B$ and $B'$, it is also true for their tensor product $B \otimes B'$. Indeed, the combinatorial analogues of the formulae \eqref{E:epphproduct} are
\begin{equation}\label{eq:trep}
\ep_i(b \otimes b') = \ep_i(b') + \max(0, \ep_i(b)-\phi_i(b')); \;\;\; \phi_i(b \otimes b') = \phi_i(b') + \max(0, \phi_i(b')-\ep_i(b)),
\end{equation}
consistent with Kashiwara's formulas for combinatorial crystals, see \cite{Sh}. Also, the tropical analogue of formula \eqref{eq:c+} is
\begin{equation*}
c^+=\min(\phi_i(b')+c, \ep_i(b))-\min(\phi_i(b'), \ep_i(b)),
\end{equation*}
which for $c = 1$ is equivalent to the Kashiwara's rules
\begin{equation} \label{eq:kashe}
\tilde e_i(b \otimes b') =
\begin{cases}
\tilde e_i(b) \otimes b' & \text{if $\phi_i(b') < \ep_i(b)$,}\\
b \otimes \tilde e_i(b') & \text{otherwise;}
\end{cases}
\end{equation}
Thus as long as an element is not mapped to zero, the the proposition holds. The last statement is well-known and holds for all crystals. 
\end{proof}

Given two consecutive factors in $B^{\tau}$ one can apply to them the {\it {combinatorial $R$-matrix}}, which is known to be a crystal automorphism.  The action of the combinatorial $R$-matrix can be realized using Sch\"{u}tzenberger's jeu de taquin as follows \cite{Sh}.  Given a pair of rectangular shaped tableaux $b_1$ and $b_2$, form a single skew semistandard tableaux by attaching them one to the other along the corner. It turns out that there is a unique pair $b_1'$ and $b_2'$ such that the shapes are swapped and the two skew tableaux jeu de taquin to the same non-skew tableaux. One defines $R(b_1 \otimes b_2) = b_1' \otimes b_2'$.
\begin{example} \label{ex:R}
Let $n=3$.
$$
\tableau[sY]{\bl&\bl&\bl&{1}&{2}&{2} \\ \bl&\bl&\bl&{2}&{3}&{3} \\ {1}&{2}&{2}}\;\; \stackrel{R}{\longrightarrow} \;\; \tableau[sY]{ \bl&\bl&\bl&{2}&{2}&{2}\\ {1}&{1}&{2} \\ {2}&{3}&{3}} \;\; \text{since both jeu de taquin to} \;\; \tableau[sY]{{1}&{1}&{2}&{2}&{2}&{2} \\ {2}&{3}&{3}}
$$
\end{example}

\begin{proposition}\label{prop:BX}
 Written in the variables $x^{(j)}$ as above, the combinatorial $R$-matrix action on $B^\tau$ is the tropicalization of the geometric $R$-matrix (Proposition \ref{prop:Rmatrix}) on $X^\tau$.
\end{proposition}

\begin{example}
In Example \ref{ex:R} we have $(x^{(1)}, x^{(2)}, x^{(3)}) = (1,2,0)$, $(y^{(1)}, y^{(2)}, y^{(3)}) = (2,0,1)$. Then tropicalizing \eqref{eq:whirlcurl} one has $$x'^{(1)} = y^{(1)} + \min(x^{(3)}, y^{(3)}) - \min(x^{(1)},y^{(1)}) = 2 + 0 - 1 = 1,$$ which agrees with the number of columns missing $1$ in $b_1'$.
\end{example}

\begin{proof}
It suffices to prove this for $X_M^2$, $X_M \times X_N$ and $X_N^2$.   (In the case of two whirl factors $X_M^2$ the statement is known, see for example \cite{HHIKTT} and \cite[Example 2]{LP2}).  One checks case by case that the functions $\ep$ and $\ph$ are preserved by the tropicalization of the geometric $R$-matrix.

Thus one may deduce from Proposition \ref{prop:Rmatrix} and \ref{prop:trop} that the  tropicalization of the geometric $R$-matrix commutes with the combinatorial crystal structure $\tilde e_i$, thus providing an automorphism of $B_M^{\otimes 2}$, $B_N^{\otimes 2}$ or $B_M \otimes B_N$.  Since the combinatorial $R$-matrix is the {\it {unique}} automorphism of a product of two Kirillov-Reshetikhin crystals \cite{Sh}, we conclude that the two must coincide.
%
\end{proof}
Proposition \ref{prop:BX} can also be deduced from Proposition \ref{prop:trop} and the uniqueness of $R$-matrix as a crystal automorphism for two Kirillov-Reshetikhin factors.

\section{Double affine geometric crystals and commutativity}\label{sec:torus}

\subsection{Orthogonal crystal and Weyl group actions}

Let $(W,W')$ and $(V,V')$ be two pairs of parallel wires in a network $\N$.  We say that $(W,W')$ and $(V,V')$ are {\it orthogonal} if they intersect in exactly four vertices $W \cap V$, $W \cap V'$, $W' \cap V$, and $W' \cap V'$.  The same definition holds when one or both of $(W,W')$ and $(V,V')$ is a pair of parallel wire cycles.

We shall say that $(W,W')$ is a pair of {\it antiparallel} wires, if they become parallel after reversing the orientation of either $W$ or $W'$.  The definition of orthogonal extends to the case where one of the pairs is antiparallel.  We shall exclude the case where both pairs are antiparallel.

Let $(W,W')$ and $(V,V')$ be two pairs of parallel wires or wire cycles.  We denote by $e_W$ and $e_V$ the respective wiring crystal actions (Section \ref{ssec:wiringcrystal}).  We also denote by $s_W$ and $s_V$ the corresponding Weyl group actions (Section \ref{ssec:Weylwhurl}), or equivalently by Lemma \ref{lem:crR} the corresponding whurl relation.  In the case that $(W,W')$ (or $(V,V')$) are wire cycles, to define the wiring crystal action we suppose we have chosen a cut.  That is, we need to pick successive intersections of $(W,W')$ with two other wires, and imagine $W$ and $W'$ is cut in between.  We say that $(W,W')$ is {\it cut by} $(V,V')$ if $(W,W')$ is a wire cycle and the chosen cut happens exactly at the intersections of $(W,W')$ with $(V,V')$.

\begin{thm}\label{thm:parallelcommute}
Suppose $(W,W')$ and $(V,V')$ are orthogonal, and both pairs are parallel.  Then
\begin{enumerate}
\item
If $(W,W')$ is not cut by $(V,V')$ and $(V,V')$ is not cut by $(W,W')$, then the wiring crystal actions $e_W$ and $e_V$ commute.
\item
If $(W,W')$ is not cut by $(V,V')$, then the wiring crystal action $e_W$ and the Weyl group action $s_V$ commute.
\item
The Weyl group actions $s_W$ and $s_V$ commute.
\end{enumerate}
\end{thm}
\begin{proof}
(2) and (3) follow from (1), since the Weyl group action is a special case of the wiring crystal action but we know from Lemma \ref{lem:crR} and Section \ref{ssec:whurl} that it does not depend on where the cut is made.

Now we prove (1).  Since neither $(W,W')$ or $(V,V')$ cuts the other, we may picture the two wiring crystal actions $e_W^c$ and $e_V^d$ as in Figure \ref{fig:wire25}.  We create four crossings, using as weights $(c-1)\ph(W,W')$ and $(d-1) \ph(V,V')$ and respectively $(c^{-1}-1)\ep(W,W')$ and $(d^{-1}-1) \ep(V,V')$ on the other side. Using the Frenkel-Moore relation (Proposition \ref{prop:fm}), we conclude that the effect on the network of pushing both crossings through does not depend on which is being pushed first. Thus, in both pairs of crossings the cancellation happens regardless of which crossing is pushed through first. 

It remains to argue that after one acts by the crystal operator $e_W$ the value of the functions $\ph(V,V')$ and $\ep(V,V')$ remains the same, so that weights $(d-1) \ph(V,V')$ and $(d^{-1}-1) \ep(V,V')$ still realize the action of the crystal operator $e_V^d$.  This is a local calculation, so we may assume that we are on the cylinder, with $(V,V')$ as horizontal wires connecting the boundaries and $(W,W')$ orthogonal wire cycles going around the cylinder.  By Proposition \ref{P:epphrat}, $\ph(V,V')$ is given by a rational function of the boundary measurements, and further by definition it is unaffected by the creation of the two crossings (weighted $(c-1)\ph(W,W')$ and $(c^{-1}-1)\ep(W,W')$) on the $(W,W')$ wires.  By Theorem \ref{thm:bm}, pushing the crossings through do not affect these boundary measurements.  It follows that $e_W$ preserves $\ph(V,V')$ and $\ep(V,V')$, and similarly $e_V$ preserves $\ph(W,W')$ and $\ep(W,W')$.
\begin{figure}[h!]
    \begin{center}
    \input{wire25.pstex_t}
    \end{center}
    \caption{}
    \label{fig:wire25}
\end{figure}
\end{proof}

Now suppose that $(V,V')$ is parallel, but $(W,W')$ is antiparallel.  Denote by $s_W$ the whirl-curl relation applied to $(W,W')$ after possibly gluing the ends of $W$ (resp. $W'$) together in the case that $(W,W')$ are not wire cycles.

\begin{thm}\label{thm:antiparallelcommute}
Suppose $(W,W')$ and $(V,V')$ are orthogonal, $(W,W')$ is antiparallel and $(V,V')$ is parallel.  Then
\begin{enumerate}
\item
If $(V,V')$ is not cut by $(W,W')$, then the wiring crystal action $e_V$ and the whirl-curl relation $s_W$ commute.
\item
The Weyl group action $s_V$ and the whirl-curl relation $s_W$ commute.
\end{enumerate}
\end{thm}
\begin{proof}
(2) follows from (1), since the Weyl group action is a special case of the wiring crystal action, but it does not depend on where the cut is made.

By Theorem \ref{thm:wcloc}, the whirl-curl relation $s_W$ can be realized by pushing a crossing through to eliminate another crossing. We apply the Frenkel-Moore relation (Proposition \ref{prop:fm}) to the situation in Figure \ref{fig:wire26} to see that it does not matter if the wiring crystal action crossing or the whirl-curl relation crossing is pushed through first. In particular, after we push through the crystal crossing, the whirl-curl relation crossings still cancel each other out. According to Lemma \ref{lem:unp} this means that one is still performing the whirl-curl transformation. Thus canceling the crossings in Figure \ref{fig:wire26} does indeed compute $s_W \circ e_V^d$.  

To see that Figure \ref{fig:wire26} also computes $e_V^d \circ s_W$, we note that the whirl-curl transformation $s_W$ (\eqref{eq:whirlcurl}) maps the ratio $x^{(i)}y^{(i+1)}/y^{(i)}x^{(i+1)}$ into its inverse, which implies that the function $$\gamma(V,V') = \frac{\prod_{i=1}^m z^{(i)}}{\prod_{i=1}^m t^{(i)}}$$ remains invariant (the orientations of $W$ and $W'$ are swapped and this cancels out the inverse).  Since the two crystal crossings cancel each other out even after the action of $s_W$, we conclude that they realize the action of $e_V^{d'}$ for some ${d'}$. The ratio of the weights of the two crystal crossings must thus be $-d'{\gamma(W,W')}^{-1}$, and since $\gamma(W,W')$ did not change, we have $d=d'$. 
\begin{figure}[h!]
    \begin{center}
    \input{wire26.pstex_t}
    \end{center}
    \caption{}
    \label{fig:wire26}
\end{figure}
\end{proof}

\subsection{The double affine geometric crystal on the torus}

Let $\N$ consist of two families of wire cycles on a torus $S$ with homology classes corresponding to a basis of $H_1(S)$, which we refer to as the {\it {vertical family}} and {\it {horizontal family}}.  Each wire cycle family intersects each wire cycle of the other family exactly once.  Assume we have $n$ horizontal wire cycles and $m$ vertical wire cycles, and that the weights of the vertices are $x_j^{(i)}$, $i \in [n]$, $j \in [m]$.  We shall suppose that all cycles in each family are oriented in the same direction, so that we have $(n+m)$ pairs of parallel wires.  Let $\XX_{n,m}$ denote the set of networks obtained by varying the weights.

There are $n$ different ways to cut the torus horizontally to form a cylinder, giving $\XX_{n,m}$ $n$ different $\uqslm$-crystal structures, and $m$ different ways to cut the torus vertically, giving $\XX_{n,m}$ $m$ different $\uqsln$-crystal structures.  We call this set of crystal structures on $\XX_{n,m}$ a {\it $\uqsln \times \uqslm$ double affine crystal}, or {\it toroidal crystal}.  We believe that this crystal is related to quantum toroidal algebras.

Note that once we have chosen $\uqslm$ and $\uqsln$ crystal structures on $\XX_{n,m}$, we have also distinguished which pair of wires correspond to the affine (0 node) crystal wiring action.  The following theorem follows from Theorem \ref{thm:parallelcommute} and Proposition \ref{prop:Rmatrix}.  

\begin{thm}\label{thm:doubleaffine}
Pick $\uqslm$ and $\uqsln$ crystal structures on $\XX_{n,m}$.  Then
\begin{enumerate}
\item
The finite ($\fsln$ and $\fslm$) crystal actions commute.
\item
The $\uqslm$ (resp. $\uqsln$) crystal action commutes with the Weyl group action for $\fsln$ (resp. $\fslm$) which acts as the R-matrix. 
\item
The (affine) Weyl group actions of the $\uqslm$ and $\uqsln$ crystals commute.
\end{enumerate}
\end{thm}

Part (3) of Theorem \ref{thm:doubleaffine} is a theorem of Kajiwara, Noumi and Yamada \cite[Theorem 2.1]{KNY} who discovered it in the context of discrete Painlev\'{e} dynamical systems.  Part (1) generalizes an observation of Berenstein and Kazhdan \cite[Example 1.4]{BK3}. 

\begin{remark}
Because of the orientation of the two families of wires, the crystal structure of $\XX_{n,m}$ as a $\uqsln$-crystal is as a product of basic affine geometric crystals, and as a $\uqslm$-crystal it is a product of dual basic affine geometric crystals (or vice versa).
\end{remark}

%

\subsection{Double affine combinatorial crystals}
Lascoux \cite{Las} has considered commuting finite crystal structures for combinatorial $U_q(\mathfrak{sl}_n)$ crystals.  In the setting of biwords, he defines two (commuting) crystal structures that act on the insertion and recording tableaux respectively. We claim that this double (finite) crystal structure is a tropicalization of the two commuting geometric crystal structures as in Theorem \ref{thm:doubleaffine}(1).

Continue the notation of Section \ref{ssec:comb}.
Let $b = b_1 \otimes b_2 \otimes \dotsc \otimes b_m \in (B^{M})^{\otimes m}$ be an element of a tensor product of $\sln$ crystals. Let $x_i^{(j)}$ be the number of $j$-s in $b_i$. Swapping the roles of $i$ and $j$ in $x_i^{(j)}$, $c = c_1 \otimes c_{n-1} \otimes \dotsc \otimes c_n \in (B^M)^{\otimes n}$ of single row crystals for $\uqslm$.  To see that the two crystal structures (on the same underlying set) commute, we first apply the duality map mentioned in the proof of Proposition \ref{prop:BX}.  This sends $c = c_1 \otimes c_{n-1} \otimes \dotsc \otimes c_n \in (B^M)^{\otimes n}$ to $c' = c'_n \otimes \cdots \otimes c'_1 \in (B_N)^{\otimes n}$, where $c'_j$ is determined by $c_j$.  Comparing with the labeling of wires on a torus and using Proposition \ref{prop:trop}, one sees that the crystal actions on $b$ and $c'$ is the combinatorial analogue of the double affine geometric crystal of Theorem \ref{thm:doubleaffine}.  The commuting finite combinatorial crystal structures then follows from Theorem \ref{thm:doubleaffine}(1).

\end{document}